\documentclass[a4paper,10pt,table]{scrartcl}
\usepackage{abstract}
\usepackage{xcolor}

\KOMAoptions{
	twoside=true,
	headsepline=false,
	headings=normal 	
}

\usepackage{titlesec}	
\titleformat{\section}{\large\bfseries\center}{\thesection}{1em}{}
\titleformat{\subsection}[runin]{\bfseries}{\thesubsection}{1em}{}
\usepackage[activeacute,spanish,es-nodecimaldot]{babel}
\usepackage[utf8]{inputenc}
\addto\captionsspanish{} % Esto hace que el caption de las figuras pase a ser Fig.
\addto\captionsspanish{}

\usepackage{amsmath, amsfonts, amssymb}
\usepackage{multicol}
\usepackage{enumerate}
\usepackage{a4wide}
\usepackage{graphicx}
\usepackage{multicol}
\usepackage[T1]{fontenc}
\usepackage[sc,osf]{mathpazo}
\usepackage[scaled]{beramono}
\usepackage{tgpagella}
\usepackage{hyperref}
\hypersetup{
    colorlinks=true, linktocpage=true, pdfstartview=FitV
    urlcolor=blue, linkcolor=blue, citecolor=blue, pagecolor=blue

}
\newtheorem{teo}{Theorem}[section]
\newtheorem{lem}[teo]{Lemma}

\newtheorem{defi}[teo]{Definition}
\newtheorem{claim}{Claim}
\newtheorem{obs}{Observation}
\numberwithin{equation}{section} % Para que las ecuaciones se numeren según la sección (3.1, 3.2, 4.1 en vez de 1,2,3).

\usepackage{fancybox,fancyhdr}
\newenvironment{proof}{\trivlist\item[\hskip \labelsep{\emph{\textbf{Proof}}}:]}{\nopagebreak \hfill $\Box$ \endtrivlist} 
\newenvironment{proofclaim}{\trivlist\item[\hskip \labelsep{\emph{\textbf{Proof of the Claim}}}:]}{\nopagebreak \hfill $\Box$ \endtrivlist}

\newcommand{\R}{\mathbb{R}}     % Reales
\def\S{\mathbb{S}}

\def\sig{\Sigma}
\def\s2{\S^2}	
\def\r3{\mathbb{R}^3}
\def\t{\theta}
\def\h5{\hspace{.5cm}}

\def\s{\mathcal{S}}
\def\p{\phi}
\def\g0{\gamma_{x_0}}

\begin{document}
\thispagestyle{empty}
%\mbox{}\vspace{0.4cm}

\begin{center}
%\rule{16cm}{1.5pt}\vspace{0.5cm}

\renewcommand{\thefootnote}{\,}
{\fontsize{16}{19} \textbf{Surfaces of prescribed linear Weingarten curvature in $\mathbb{R}^3$} 
\footnote{
\hspace{.15cm}\emph{Mathematics Subject Classification:} 53A10, 53C42, 34C05, 34C40
$^{\dagger,\ddagger}$Departamento de Ciencias, Centro Universitario de la Defensa de San Javier, E-30729 Santiago de la Ribera, Spain.\\
\emph{e-mail address:} $^\dagger$antonio.bueno@cud.upct.es, $^\ddagger$irene.ortiz@cud.upct.es}}

\vspace{0.5cm} { Antonio Bueno$^\dagger$, Irene Ortiz$^\ddagger$}\\
	%\rule{16cm}{1.5pt}
\end{center}
\vspace{.5cm}
\begin{quotation}
\emph{This paper is dedicated to the first author's mother, whose daily fight, strength and spirit of overcoming adversity, proves that there are way more difficult and important issues in life than a mathematical problem}
\end{quotation}
\vspace{.5cm}

\begin{abstract}
Given $a,b\in\R$ and $\Phi\in C^1(\S^2)$, we study immersed oriented surfaces $\sig$ in the Euclidean 3-space $\R^3$ whose mean curvature $H$ and Gauss curvature $K$ satisfy $2aH+bK=\Phi(N)$, where\\ $N:\sig\rightarrow\S^2$ is the Gauss map. This theory widely generalize some of paramount importance such as the ones constant mean and Gauss curvature surfaces, linear Weingarten surfaces and self-translating solitons of the mean curvature flow. Under mild assumptions on the prescribed function $\Phi$, we exhibit a classification result for rotational surfaces in the case that the underlying fully nonlinear PDE that governs these surfaces is elliptic or hyperbolic.
\end{abstract}

\section{Introduction}
In this paper we study the existence and classification of rotational surfaces $\sig$ in the Euclidean 3-space $\R^3$ whose mean curvature $H$ and Gauss curvature $K$ satisfy a linear relation that depends on their Gauss map $N:\sig\rightarrow\S^2$. Specifically, given $\Phi\in C^1(\S^2)$ and $a,b\in\R$, we are interested in surfaces satisfying
\begin{equation}\label{eq:defiPLW}
2aH+bK=\Phi(N).
\end{equation}
Any such $\sig$ will be called a \emph{prescribed linear Weingarten curvature surface}, or $\Phi$-\emph{surface} for short. Depending on $a,b$ and $\Phi$, numerous different examples of $\Phi$-surfaces have already appeared in the literature. Next, we highlight some among the most relevant:
\begin{itemize}
\item If $a=0,\Phi=c\in\R$, we have surfaces of constant Gauss curvature. In general, for an arbitrary $\Phi$, we have surfaces of prescribed Gauss curvature.
\item If $b=0,\Phi=c\in\R$, we have surfaces of constant mean curvature. In general, for an arbitrary $\Phi$, we have surfaces of prescribed mean curvature. An important particular case is when $\Phi(X)=\langle X,e_3\rangle,\ \forall X\in\S^2$, for which we recover self-translating solitons of the mean curvature flow. See e.g. \cite{CSS,HuSi,Ilm,MSHS} and references therein.
\item If $\Phi=c\in\R$ and $ab\neq0$, we have linear Weingarten surfaces. See e.g. \cite{Che,GMM,HaWi2,Hop,Lop,RoSa,SaTo1,SaTo2} and references therein.
\end{itemize}
Hereinafter, we will always assume that $a,b$ are not null in order to avoid these already studied cases.

In general, the study of surfaces in $\R^3$ described by a curvature function in terms of their Gauss map goes back, at least, to the Minkowski and Christoffel problems for ovaloids \cite{Chr,Min}. In the first one, the Gauss curvature is prescribed, while in the second one it is the sum of the curvature radii $1/k_1+1/k_2$. When the prescribed curvature function is the mean curvature, the existence and uniqueness of ovaloids was approached, among others, by Alexandrov, Pogorelov, Hartman and Wintner \cite{Ale, Pog,HaWi}, and more recently by B. Guan, P. Guan, Gálvez and Mira \cite{GuGu,GaMi1,GaMi2}. The main theorem in \cite{GaMi2} provides a tremendous general uniqueness result for immersed spheres in 3 manifolds, that widely generalizes Hopf's uniqueness theorem for constant mean curvature spheres. In our framework, it states that if there exists a sphere satisfying \eqref{eq:defiPLW}, and this equation is elliptic, it is unique among immersed spheres also satisfying \eqref{eq:defiPLW}. However, less is known about complete, non-compact $\Phi$-surfaces for a general function $\Phi\in C^1(\S^2)$, besides the paramount, aforementioned theories of constant mean and Gauss curvature, self-translating solitons of the mean curvature flow and linear Weingarten surfaces. In this fashion, the first author jointly with Gálvez and Mira started to develop the \emph{global theory of surfaces of prescribed mean curvature}, taking as starting point the theory of positive constant mean curvature surfaces and self-translating solitons of the mean curvature flow \cite{BGM1,BGM2}. Also, the authors recently addressed the study of surfaces in $\R^3$ of prescribed Gauss curvature, starting by describing the rotational examples \cite{BuOr}. 

Our goal in this paper is twofold. First, we pursue a complete, self-contained classification of rotational linear Weingarten surfaces, i.e. when $\Phi=c\in\R$ in Eqn. \eqref{eq:defiPLW}. Different authors have approached this problem, distinguishing between the character of the underlying PDE that locally governs Eqn. \eqref{eq:defiPLW} as elliptic \cite{RoSa,SaTo1} and hyperbolic \cite{Lop}, but either some examples were missed or the classification result had some mistake, respectively. In this paper, we cover these gaps and exhibit the explicit behavior of the profile curve of every rotational linear Weingarten surface. For a description of the parabolic case, we refer the reader to \cite{BuLo}.

Second, we aim to lay the groundwork of a global theory of complete, non-compact $\Phi$-surfaces, taking as main motivation the theories of constant mean and Gauss curvature surfaces, and also the one of linear Weingarten surfaces. In virtue of Eqn. \eqref{eq:defiPLW}, the following are two trivial properties of $\Phi$-surfaces: (1) $\Phi$-surfaces are invariant by Euclidean translations; and (2) any symmetry of $\Phi$ in $\S^2$ induces a linear isometry of $\R^3$ that preserves the class of $\Phi$-surfaces, i.e. that send $\Phi$-surfaces into $\Phi$-surfaces. Another fundamental property of $\Phi$-surfaces is that they can be locally expressed as the graph of a function $u(x,y)$ that is a solution of the fully nonlinear PDE
\begin{equation}\label{eq:PDE}
a\,\mathrm{div}\left(\frac{Du}{\sqrt{1+|Du|^2}}\right)+b\frac{det D^2u}{(1+|Du|^2)^2}=\Phi\left(\frac{(-Du,1)}{\sqrt{1+|Du|^2}}\right),
\end{equation}
where $\mathrm{div},D$ and $D^2$ are, respectively, the divergence, gradient and hessian operators on $\R^3$. Some difficulties that arise when approaching the study of \eqref{eq:PDE} are: (1) for a general non-constant function $\Phi$, this equation does not have a variational structure; and (2) due to the arbitrariness of $\Phi$ we need to take into account the loss of symmetries and isotropy of the resulting equation We remark that for the particular case $\Phi=c$, a linear Weingarten surface $\sig$ is a critical point of a functional expressed as a linear combination of the area of $\sig$, the volume enclosed by $\sig$, and the total mean curvature of $\sig$. 

The line of inquiry in the present paper is the existence and classification of rotational $\Phi$-surfaces in the case that the prescribed function $\Phi$ is \emph{rotationally symmetric}, i.e. $\Phi(N)=\phi(\langle N,e_3\rangle)$, where $\phi\in C^1([-1,1])$. Under this symmetry condition, Eqn. \eqref{eq:defiPLW} expresses as
\begin{equation}\label{eq:defiPLWrotsim}
2aH+bK=\phi(\langle N,e_3\rangle).
\end{equation}
The quantity $\langle N,e_3\rangle$ that measures the height of $N$ in $\S^2$ with respect to the $e_3$-direction is the so-called \emph{angle function} of the $\Phi$-surface. Under this hypothesis, rotations in $\R^3$ around a vertical line are ambient isometries that preserve Eqn. \eqref{eq:defiPLWrotsim}, hence the notion of rotational $\Phi$-surface is well-defined.

%For the particular case that $\Phi=c\in\R$, rotational linear Weingarten surfaces have been analyzed by different authors, depending on the character of the underlying PDE of Eqn. \eqref{eq:defiPLW} as elliptic, hyperbolic or parabolic. In the elliptic case, rotational linear Weingarten surfaces appeared in \cite{RoSa} and were partially classified in \cite{SaTo2} in the more general setting of \emph{special surfaces}. In the latter work, the authors realized that rotational special surfaces roughly behaved as surfaces of constant mean and positive Gauss curvature. Nonetheless, a classification of all the rotational linear Weingarten surfaces of elliptic type has been still an open issue. As a particular consequence of our investigations, in this paper we provide the full classification of all rotational linear Weingarten surfaces of elliptic type, settling this problem. In the hyperbolic case, the classification of rotational linear Weingarten surfaces was addressed by López \cite{Lop}. For the parabolic case, see the work of the first author and López \cite{BuLo} for an explicit description of its rotational examples.

A classification result for \emph{all} $\Phi$-surfaces with no further hypotheses on $\Phi\in C^1(\S^2)$ seems hopeless in such generality. When $b=0$ in \eqref{eq:defiPLW}, rotational surfaces of prescribed mean curvature $H=\mathcal{H}(N)$ were studied \cite{BGM2}. Under the hypotheses on $\mathcal{H}$ of being non-vanishing and even, the rotational $\mathcal{H}$-surfaces are open pieces of spheres, cylinders, unduloids and nodoids. This \emph{Delaunay pattern} of $\mathcal{H}$-surfaces has been extended by the first author to more general ambient spaces: the so-called $\mathbb{E}(\kappa,\tau)$ spaces, see \cite{Bue1,Bue2,Bue3} and references therein. In the same fashion, taking $a=0$ in \eqref{eq:defiPLW}, the study of rotational surfaces of prescribed Gauss curvature $K=\mathfrak{K}(N)$ was addressed by the authors in \cite{BuOr}. Again, the hypotheses $\mathfrak{K}\neq0$ and even revealed as sufficient to generalize the classification of rotational surfaces of constant Gauss curvature. Inspired by the well structure induced by these conditions, throughout this paper we will assume that $\Phi$ is a rotationally symmetric, non-vanishing, even function on $\S^2$. In these conditions, reflections about horizontal planes preserve Eqn. \eqref{eq:defiPLWrotsim} and hence send $\Phi$-surfaces into $\Phi$-surfaces.

Besides the hypotheses on $\Phi$ of being non-vanishing and even, it is also necessary to distinguish between the intrinsic character of the PDE \eqref{eq:PDE}, namely elliptic, hyperbolic or parabolic; this character will determine the global structure of $\Phi$-surfaces, similarly to what happens for linear Weingarten surfaces. As revealed in Section \ref{sec22}, for a general $\Phi$-surface, the character of the PDE \eqref{eq:PDE} is determined by the sign of its discriminant $a^2+b\Phi$. We furthermore emphasize that the parabolic case, namely when $a^2+b\Phi=0$, leads to $\Phi$ being a constant function and has been solved in \cite{BuLo}. Consequently, in this paper, we study in detail the classification of elliptic and hyperbolic rotational $\Phi$-surfaces. 

%Distinguishing between these cases and the further hypotheses on $\Phi$ of being even and non-vanishing are sufficient in order to provide a full classification for rotational $\Phi$-surfaces.

We next explain the organization of the paper, and highlight some of the main results. In Section \ref{sec2}, we analyze in detail the nonlinear autonomous system fulfilled by the profile curve of a rotational $\Phi$-surface. In Section \ref{sec21}, we carry on a qualitative study of its solutions that allows us to deduce the behavior and the geometric properties of such profile curves, and eventually to classify the rotational examples. In Section \ref{sec22} we deduce the local character of the PDE \eqref{eq:PDE} as elliptic, hyperbolic and parabolic, in terms of $a,b$ and $\Phi$. In Section \ref{sec23}, we briefly discuss the existence of radial solutions of Eqn. \eqref{eq:PDE} intersecting orthogonally the axis of rotation. This equation for such initial data is singular, hence standard theory cannot be invoked in order to ensure its existence. We overcome these difficulties in virtue of the work of the first author and López \cite{BuLo}, and deduce its main consequence on the phase plane.

In Section \ref{sec3}, we classify rotational $\Phi$-surfaces of elliptic type. We distinguish two possible cases on the constants $a,b$ in Sections \ref{sec31} and \ref{sec32}, and prove Thms. \ref{th:clasificacionespeciales} and \ref{th:clasificacionelipticas2}, respectively. We deduce the existence of 10 distinct examples of rotational $\Phi$-surfaces of elliptic type: some resemble to constant mean curvature surfaces, while others to positive constant Gauss curvature surfaces.

Finally, in Section \ref{sec4}, we classify all the rotational $\Phi$-surfaces of hyperbolic type in Thm. \ref{th:clasificacionhiperbolicas}. As a particular case, we correct a mistake in the classification result in \cite{Lop} of rotational linear Weingarten surfaces of hyperbolic type.

\section{Differential equations of rotational $\Phi$-surfaces}\label{sec2}

Let $\Phi\in C^1(\S^2)$ be rotationally symmetric around the $e_3$-direction, i.e. there exists a 1-dimensional function $\p\in C^1([-1,1])$ such that
\begin{equation}\label{condicionrotsim}
\Phi(X)=\phi(\langle X,e_3\rangle),\hspace{.5cm} \forall X\in\S^2.
\end{equation}
A surface satisfying \eqref{eq:defiPLW} for such $\Phi$, or equivalently \eqref{eq:defiPLWrotsim} for $\phi$, and $a,b\in\R$ is called a $\Phi$-surface. Our first step is to deduce the differential equations fulfilled by the coordinates of the profile curve of a rotational $\Phi$-surface. 

Let $\sig$ be a $\Phi$-surface described as the rotation of an arc-length parametrized, planar curve $\alpha(s)=(x(s),0,z(s)),\ s\in I\subset\R$, around the vertical axis passing through the origin. From the arc-length condition we deduce the existence of a function $\theta(s)$ defined by $x'(s)=\cos\t(s),\ z'(s)=\sin\t(s)$. Hereinafter and unless explicitly said and needed, we omit the parameter $s$ for saving notation. A straightforward computation yields that the principal curvatures of $\sig$ are
\begin{equation}\label{eq:curvprin}
\kappa_1=\theta',\hspace{.5cm} \kappa_2=\frac{\sin\theta}{x},
\end{equation}
and so the Gauss and mean curvature are
$$
K=\frac{\theta'\sin\theta}{x},\hspace{.5cm} 2H=\theta'+\frac{\sin\theta}{x}.
$$
Moreover, the angle function, $\langle N,e_3\rangle$, is $x'=\cos\t$. Bearing these discussions in mind, Eqn. \eqref{eq:defiPLWrotsim} transforms into the differential equation
$$
a\left(\theta'+\frac{\sin\theta}{x}\right)+b\frac{\theta'\sin\theta}{x}=\p(\cos\theta).
$$
Solving $\theta'$ we arrive to
$$
\theta'=\frac{x\p(\cos\theta)-a\sin\theta}{ax+b\sin\theta},
$$
wherever $ax+b\sin\t\neq0$. Conversely, assume that $x(s),z(s)$ and $\t(s)$ are solutions of
\begin{equation}\label{eq:sistemadiferencial}
\left\lbrace\begin{array}{l}
\vspace{.25cm} x'=\cos\theta\\
\vspace{.25cm} z'=\sin\theta\\
\theta'=\displaystyle{\frac{x\p(\cos\theta)-a\sin\theta}{ax+b\sin\theta}},
\end{array}\right.
\end{equation}
$ax+b\sin\t\neq0$, for some $\phi\in C^1([-1,1])$. Then, the surface defined by rotating $\alpha(s)=(x(s),0,z(s))$ around the $z$-axis is a $\Phi$-surface for the function $\Phi(X)=\p(\langle X,e_3\rangle),\ \forall X\in\S^2$.

Recall that we can reduce system \eqref{eq:sistemadiferencial} to the first and third equations, since the second one is defined by the other two. Consequently, we project the vector field $(x,z,\t)$ onto the $(x,\t)$-plane, obtaining the first order, nonlinear autonomous system
\begin{equation}\label{eq:planofases}
\left(\begin{array}{c}
x\\
\theta
\end{array}\right)'=
\left(\begin{array}{c}
\vspace{.1cm} \cos\theta(s)\\
\displaystyle{\frac{x\p(\cos\theta)-a\sin\theta}{ax+b\sin\theta}}
\end{array}\right).
\end{equation}

\subsection{The phase plane.}\label{sec21} The phase plane is defined as the set $\Theta:=(0,\infty)\times(0,2\pi)-\s$, where
$$
\s:=\left\lbrace(x,\t):\ x>0,\ \t\in(0,2\pi),\ x=\s(\theta):=\frac{-b\sin\t}{a}\right\rbrace.
$$
The orbits $\gamma(s)=(x(s),\theta(s))$ are the solutions of system \eqref{eq:planofases}. Next, we summarize some of the main properties of the phase plane that follow from its definition:
\begin{enumerate}
\item We define $\Theta_1=\Theta\cap\{\theta<\pi\}$ and $\Theta_2=\Theta\cap\{\t\in(\pi,2\pi)\}$. Since $x>0$, $\s$ is a compact arc in $\Theta_1$ (resp. in $\Theta_2$) joining the points $(0,0)$ and $(0,\pi)$ (resp. $(0,\pi)$ and $(0,2\pi)$) if and only if $ab<0$ (resp. if $ab>0$).
\item The points lying on $\s$, are singular for system \eqref{eq:planofases}. Therefore, an orbit cannot have a finite endpoint at $\s$, although it can converge to a point in $\s$ at finite time. Any such limit point in $\s$ corresponds to a circle of singular points of the corresponding $\Phi$-surface.
\item By solving $\theta'=0$ in \eqref{eq:planofases}, we deduce that the curve $\Gamma:=\Theta\cap\{x=\Gamma(\t)\}$, where
$$
x=\Gamma(\t):=\frac{a\sin\t}{\p(\cos\t)},\h5 \t\in(0,2\pi),\ \phi(\cos\t)\neq0,
$$
corresponds to points of the profile curve $\alpha$ with vanishing curvature. This curve exists in $\Theta$ wherever $a\sin\t\p(\cos\t)>0$.
\item If $\phi(0)\neq0$, the point
$$
e_0=\left(\frac{a}{\phi(0)},\frac{\pi}{2}\right),\ \mathrm{if}\ a\phi(0)>0,\hspace{.5cm} \left(\frac{-a}{\phi(0)},\frac{3\pi}{2}\right),\ \mathrm{if}\ a\phi(0)<0,
$$
called the \emph{equilibrium}, is a constant orbit that trivially solves \eqref{eq:planofases}. The $\Phi$-surface generated by this orbit is the circular flat cylinder of vertical rulings and radius $|a/\phi(0)|$.
\item The lines $\theta=\pi/2,3\pi/2$ and the curves $\Gamma,\s$ divide $\Theta$ into connected components where the coordinate functions of an orbit are monotonous.
\end{enumerate}

Throughout this paper we will always assume that the prescribed function $\Phi\in C^1(\S^2)$ is positive and \emph{even}, that is $\Phi(X)=\Phi(-X)$, for every $X\in\S^2$. In particular, $\phi$ given by \eqref{condicionrotsim} in terms of $\Phi$ satisfies $\phi(y)=\phi(-y)$ for every $y\in[-1,1]$. The even condition has the following geometric property on the phase plane: if $(x(s),\t(s))$ is a solution of \eqref{eq:planofases}, then $(x(-s),\pi-\t(-s))$ is also a solution of \eqref{eq:planofases} whenever $\t\in(0,\pi)$. When $\t\in(\pi,2\pi)$, it follows that $(x(-s),3\pi-\t(-s))$ is a solution to \eqref{eq:planofases}. Geometrically, the subsets $\Theta_1$ and $\Theta_2$ of $\Theta$ are symmetric with respect to the lines $\t=\pi/2,\ \t=3\pi/2$, respectively.

%\section{Prescribed special linear Weingarten surfaces}

\subsection{The local character of the PDE}\label{sec22}

Locally, a surface satisfying \eqref{eq:defiPLW} can be expressed as the graph of a function $u(x,y) $ that is a solution of the PDE \eqref{eq:PDE}, which can be written as the solution of a function $\mathfrak{F}(u_x,u_y,u_{xx},u_{yy},u_{xy})=0$. Note that the variable $u$ does not appear, since \eqref{eq:PDE} is invariant under Euclidean translations. A straightforward computation shows that the discriminant of this PDE is
$$
\mathfrak{F}_{u_{xx}}\mathfrak{F}_{u_{yy}}-\frac{1}{4}\mathfrak{F}_{u_{xy}}^2=(1+u_x^2+u_y^2)^2(a^2+b\Phi).
$$
The character of this PDE inspires the following definition.
\begin{defi}\label{tipoEDP}
Let be $a,b\in\R$ and $\Phi\in C^1(\S^2)$ related by \eqref{eq:defiPLW}, and $\sig$ a $\Phi$-surface.
\begin{itemize}
\item If $a^2+b\Phi>0$, $\sig$ is of \emph{elliptic} type.
\item If $a^2+b\Phi=0$, $\sig$ is of \emph{parabolic} type.
\item If $a^2+b\Phi<0$, $\sig$ is of \emph{hyperbolic} type.
\end{itemize}
\end{defi}
For instance, if $\sig$ is a $\Phi$-surface of elliptic type then the maximum principle for fully nonlinear PDE's applies. Consequently, classic techniques originally developed for surfaces of constant mean curvature and positive constant Gauss curvature can be adapted to this case, see e.g. \cite{RoSa}. The parabolic case leads to $\Phi$ being a constant function, and an explicit description of the rotational parabolic linear Weingarten surfaces can be found in \cite{BuLo}.

\subsection{Existence of radial solutions}\label{sec23} In this section we discuss the existence of radial solutions, $u=u(r)$, of Eqn. \eqref{eq:PDE} that intersect orthogonally the axis of rotation. This is equivalent to establish a classical solution of
\begin{equation}\label{rot}
\left\{\begin{array}{ll}a\left(\dfrac{u''}{(1+u'^2)^{3/2}}+\dfrac{u'}{r\sqrt{1+u'^2}}\right)+b\dfrac{u''u'}{r(1+u'^2)^2}=\phi\left(\dfrac{1}{\sqrt{1+u'^2}}\right), & \mbox{in $(0,\delta)$}\\
u'(0)=0.
\end{array}\right.
\end{equation} 
Any solution of this initial value problem generates an orbit of system \eqref{eq:planofases} having the point $(0,0)$ as endpoint. However, Eqn. \eqref{rot} is singular at $r=0$, as well as Eqn. \eqref{eq:planofases} at $(0,0)$, hence standard theory cannot be invoked in order to ensure the existence of such an orbit. Nevertheless, by the work of the first author and López, for every $a,b,\Phi$ such that the elliptic relation $a^2+b\Phi>0$ holds, there exists $\delta>0$ small enough and $u\in C^2([0,\delta]),\ u>0,$ that solves \eqref{rot}. We give a hint of the method used to prove it, see Th. 4.1 in \cite{BuLo} for further details.

Assume that $u=u(r)$ is a radial solution of \eqref{eq:PDE}. After multiplying by $r$ and dividing by $a$, the equation in \eqref{rot} transforms into
$$
\left(\frac{ru'}{\sqrt{1+u'^2}}\right)'+\frac{b}{2a}\left(\frac{u'^2}{1+u'^2}\right)'=\frac{r}{a}\phi\left(\frac{1}{\sqrt{1+u'^2}}\right).
$$
If we define the functions $f,g\colon\R \to\R$ by  
$$
f(y)=\frac{y}{\sqrt{1+y^2}},
\quad g(y)=\frac{1}{a}\phi\left(\frac{1}{\sqrt{1+y^2}}\right),
$$
and after integration, we get
$$ 
rf(u')+\frac{b}{2a}f(u')^2=\int_0^r tg(u'(t))\, dt.
$$
After solving $f(u')$ and eventually $u'$ by taking $f^{-1}$ and integrating, we define the operator
$$
(\mathsf{T}u)(r)=\int_0^rf^{-1}\left(\frac{2a}{b}\left(-s+\sqrt{s^2+\frac{b}{a}\int_0^stg(u'(t))\, dt}\right)\right)\, ds.
$$
Note that $u(r)$ is a solution of \eqref{rot} if and only if it is a fixed point of $\mathsf{T}$. It was proved in Th. 4.1 in \cite{BuLo} that $\mathsf{T}$ is a contraction in the space $C^1([0,\delta])$ for $\delta>0$ small enough, hence the existence of a fixed point $u\in C^1([0,\delta])\cap C^2((0,\delta])$ that solves Eqn. \eqref{rot} is ensured. The fact that $u(r)$ has $C^2$-regularity at $r=0$ follows by taking limits in \eqref{rot} and applying the L'Hôpital rule.

By the even condition on $\Phi$, we can reflect $u$ with respect a horizontal plane and obtain a downwards oriented graph intersecting orthogonally the axis of rotation. In particular, we derive the following consequence on the phase plane.

\begin{lem}\label{orbitasmasmenos}
Let be $a,b\in\R$ and $\Phi\in C^1(\S^2)$ such that $a^2+b\Phi>0$. Then, there exists an orbit $\gamma_+$ (resp. $\gamma_-$) in $\Theta$ having the point $(0,0)$ (resp. $(0,\pi)$) as endpoint. 
\end{lem}

\section{The elliptic case}\label{sec3}
In this section, we classify rotational $\Phi$-surfaces of elliptic type in the sense of Def. \ref{tipoEDP}, i.e. the constants $a,b$ and the prescribed function $\Phi\in C^1(\S^2)$ related by Eqn. \eqref{eq:defiPLW} satisfy $a^2+b\Phi>0$. For the particular case $\Phi=c\in\R$, these surfaces appeared as a particular case of the more general \emph{special surfaces}, see the pioneer groundwork \cite{RoSa,SaTo1,SaTo2}. Although rotational linear Weingarten surfaces of elliptic type are special surfaces, hence were partially described in these works, some examples were missing.

In this section we provide a classification result for rotational $\Phi$-surfaces when $\Phi$ is non-vanishing and even. In particular, we fully classify rotational linear Weingarten surfaces of elliptic type, settling this problem. As a consequence of our investigations, we obtain a total of 10 distinct examples of such rotational $\Phi$-surfaces of elliptic type, described in the following table depending on $a,b$.

%We remark that for the particular case when $\Phi\equiv c\in\R,\ c>0$, these surfaces appeared in the more global context of \emph{special Weingarten surfaces}, addressed by Rosenberg and sa Earp \cite{??}. They studied the global structure of properly embedded, non compact surfaces, generalizing well-known results due to Meeks, Korevaar, Kusner and Solomon for constant mean curvature surfaces in $\R^3$ and $\mathbb{H}^3$. 

\begin{center}
\begin{tabular}{ |p{5cm}|p{5cm}| }
\hline
\multicolumn{2}{|c|}{\textbf{Rotational surfaces of elliptic type}} \\
\hline
\hfil Case $a>0,b>0$ & \hfil Case $a>0,b<0$  \\ 
\hline
\rowcolor{cyan}
\multicolumn{2}{|c|}{\textbf{Complete}} \\
\hline
\multicolumn{2}{|c|}{Cylinder} \\
\multicolumn{2}{|c|}{Sphere} \\
\multicolumn{2}{|c|}{Unduloids} \\
\multicolumn{2}{|c|}{Nodoids} \\
\hline
\rowcolor{cyan}
\multicolumn{2}{|c|}{\textbf{Non-complete}} \\
\hline
\multicolumn{2}{|c|}{$K>0$; two cusp points; strictly monotonous height} \\
\multicolumn{2}{|c|}{$K>0$; annuli; strictly monotonous height} \\
\hline
$K>0$; annuli; non-monotonous height & $K>0$; two cusp points; non-monotonous height\\
\hline
 $K$ changing sign; two cusp points; strictly monotonous height & $K$ changing sign; annuli; strictly monotonous height\\
\hline
\end{tabular}
\end{center}
Let us briefly explain the information in this table. First, we show six common examples for $b>0$ and $b<0$: four complete and two non-complete. Recall that the complete ones have the same properties as the \emph{Delaunay surfaces} of positive, constant mean curvature, and also for the case $b=0$ and an arbitrary positive and even $\Phi$ \cite{BGM2}. Next, we describe the corresponding examples according to $b>0$ or $b<0$. We exhibit if the Gauss curvature of the surface has constant sign or changes it. Then, we describe the topology: it can be either simply connected with two singular cusp points at the axis of rotation, or be an annulus whose boundary are circles of singular points. Last, we distinguish on the height function being strictly monotonous or attaining local extrema. 

At first sight, one may guess that the cases $a<0,b>0$ and $a<0,b<0$ are missing in this study. Nevertheless, we show that for fixed $b$, the cases $a>0$ and $a<0$ are the same. Indeed, since $\Phi$ is even in $\S^2$ and changing the Gauss map $N$ by $-N$ changes the sign of the mean curvature but remains invariant the Gauss curvature, the constant $a$ can be changed of sign in Eqn. \eqref{eq:defiPLW} after an adequate change of the orientation. Hence, we restrict ourselves to the cases $a>0,b>0$ and $a>0,b<0$. Remind that the only hypothesis on $\Phi$ are to be non-vanishing and even. We show next that we can further assume the prescribed function $\Phi$ to be positive. Indeed, if it is negative we consider the equivalent equation
$$
-2aH-bK=-\Phi(N),
$$
where $-\Phi$ is a positive function. If $-a>0,-b>0$ or $-a>0,-b<0$ we are in the aforementioned cases of study. If this is not the case, we change the orientation (hence the sign of $a$) and we are done. So, we assume once and for all $a>0,\Phi>0$ and distinguish between $b$ being positive or negative. 

\subsection{The case $a>0,b>0$.}\label{sec31}
The classification result obtained in this section is the following:
\begin{teo}\label{th:clasificacionespeciales}
Let be $a>0,b>0$ and $\Phi\in C^1(\S^2)$ positive and even. 

A complete, rotational $\Phi$-surface is one of the following:
\begin{enumerate}
\item[1.1] A vertical, circular cylinder;
\item[1.2] a strictly convex sphere;
\item[1.3] a properly embedded surface of unduloid-type; and
\item[1.4] a properly immersed surface of nodoid-type, with self intersections.
\end{enumerate}

Besides open pieces of the above examples, a non-complete, rotational $\Phi$-surface is one of the following:
\begin{enumerate}
\item[1.5] a surface of $K>0$, with two cusp points at the $z$-axis and strictly monotonous height;
\item[1.6] a surface of $K>0$, homeomorphic to an annulus and with strictly monotonous height;
\item[1.7] a surface of $K>0$, homeomorphic to an annulus and with non-monotonous height; and
\item[1.8] a surface of $K$ changing sign, with two cusp points at the $z$-axis and strictly monotonous height.
\end{enumerate}
%Moreover, surfaces $1.3-1.8$ can be parametrized by their distances to the axis of rotation.
\end{teo}

\begin{proof}
Let $\phi\in C^1([-1,1])$ the 1-dimensional function defined by Eqn. \eqref{condicionrotsim} in terms of $\Phi$. In particular, $\p$ is positive and even, and since $a,b>0$, the structure of the phase plane is as follows (see Fig.\ref{fig:planofaseselipticas}):
\begin{enumerate}
\item The curve $\s$ lies entirely in $\Theta_2$ and is a compact arc joining the points $(0,\pi)$ and $(0,2\pi)$. Such a curve does not exist in $\Theta_1$
\item The curve $\Gamma$ lies entirely in $\Theta_1$ and is a compact arc joining the points $(0,0)$ and$ (0,\pi)$. Such a curve does not exist in $\Theta_2$.
\end{enumerate}

\begin{figure}[h]
\centering
\includegraphics[width=.5\textwidth]{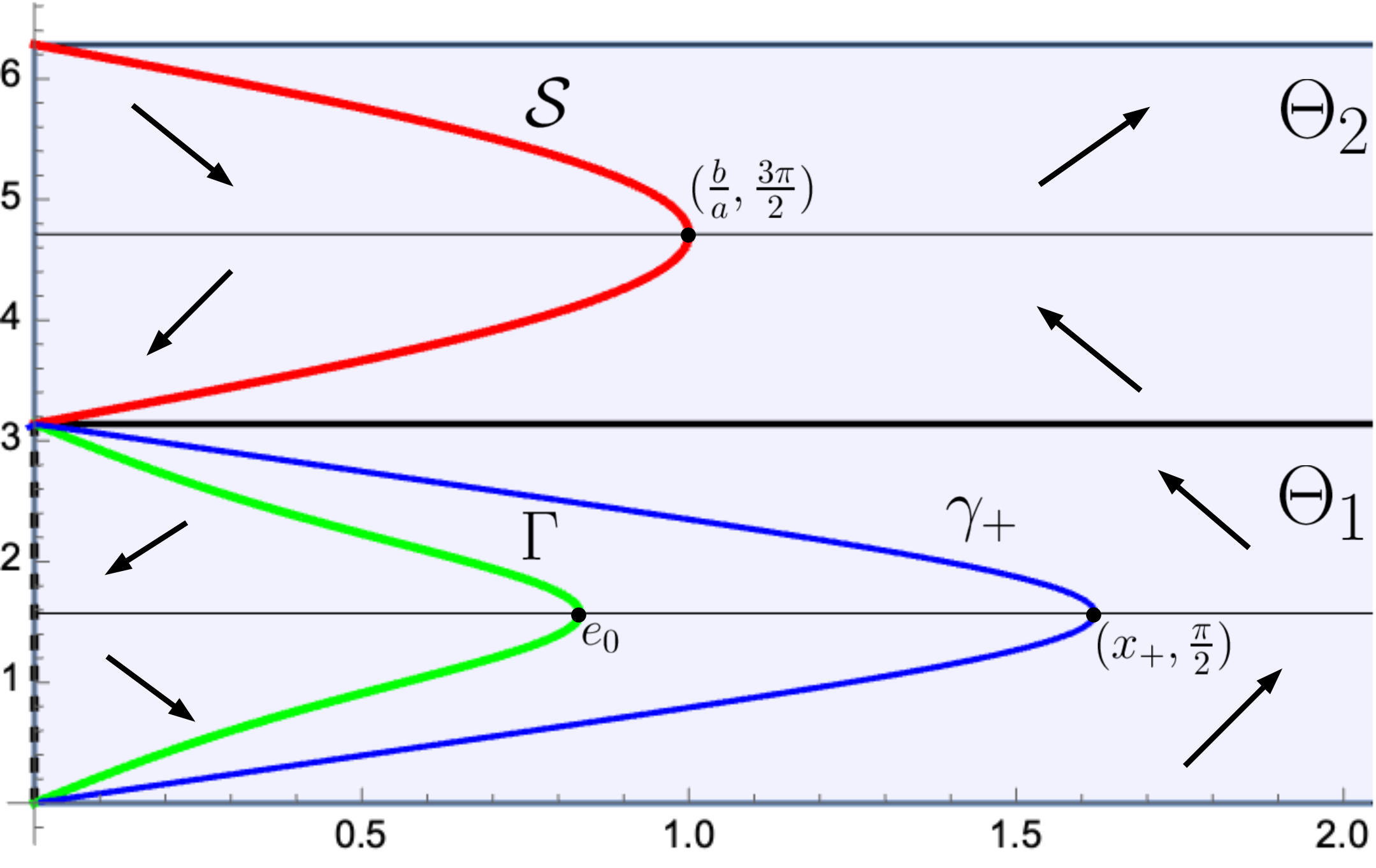}
\caption{The phase plane for the case $a>0,b>0$ and the orbit of the rotational sphere in blue.}
\label{fig:planofaseselipticas}
\end{figure}

We prove first the existence of the complete examples. Since $\Gamma(\pi/2)=a/\p(0)$, the equilibrium $e_0=(a/\p(0),\pi/2)$ exists in $\Theta_1$ and generates a circular, flat cylinder of constant mean curvature $\p(0)/(2a)$, obtaining the first example in the classification.

To prove the existence of a strictly convex sphere, consider the orbit $\gamma_+(s)=(x_+(s),\t_+(s))$ having $(0,0)$ as endpoint, given by Lemma \ref{orbitasmasmenos}, and the corresponding rotational $\Phi$-surface, $\sig_+$. First, we prove that $x_+(s)\rightarrow\infty,\t_+(s)\rightarrow\t_0\in(0,\pi/2]$ cannot happen.

Arguing by contradiction, suppose that $x_+(s)\rightarrow\infty$ and $\t_+(s)\rightarrow\t_0\in(0,\pi/2]$, which implies that $\gamma_+$ is strictly contained in the monotonicity region $\{x>\Gamma(\t),\ \t<\pi/2\}$. In particular, the function $\t_+(s)$ can be seen as a function of $x_+(s)$ and the chain rule yields
$$
\t_+'(s)=\frac{d\t_+}{ds}=\frac{d\t_+}{dx_+}\frac{dx_+}{ds}=\t_+'(x)\cos\t_+(x).
$$
Therefore, $\t_+(x)\rightarrow\t_0\in(0,\pi/2]$ as $x\rightarrow\infty$, and so $\t_+'(x)\rightarrow0$ as $x\rightarrow\infty$. But since $\t_+'(x)\cos\t_+(x)=\t_+'(s)$, we have
$$
\t_+'(x)\cos\t_+(x)=\frac{x\p(\cos\t_+(x))-a\sin\t_+(x)}{ax+b\sin\t_+(x)},
$$
and taking limits as $x\rightarrow\infty$ we get $0=\p(\t_0)/a$, a contradiction. 

So, $\gamma_+$ starts at $(0,0)$ and cannot stay forever in $\{x>\Gamma(\t),\ \t<\pi/2\}$. Next, we see that $\gamma_+$ cannot converge to $e_0$. In fact, we prove that the orbits close enough to $e_0$ are ellipses enclosing $e_0$ in their inner regions. The linearized system of \eqref{eq:planofases} around $e_0$ is given by
$$
\left(\begin{matrix}
u\\
v
\end{matrix}\right)'=
\left(\begin{matrix}
0&-1\\
\frac{a^2+b\p(0)}{(b+\frac{a^2}{\p(0)})^2}&-\frac{a\p'(0)}{a^2+b\p(0)}
\end{matrix}\right)\left(\begin{matrix}
u\\
v
\end{matrix}\right).
$$
Noting that $\p'(0)=0$ since $\p$ is even, we check that the orbits of the linearized are ellipses around the origin. By classical theory of non-linear autonomous systems, there are two possible configurations for the orbits of \eqref{eq:planofases} around $e_0$: either all such orbits are closed curves having $e_0$ in their inner regions; or they spiral around $e_0$. However, the latter possibility cannot happen since the orbits of \eqref{eq:planofases} are symmetric with respect to the line $\t=\pi/2$. In particular, all the orbits stay at a positive distance to $e_0$.

Since $\gamma_+$ cannot stay forever in the monotonicity region $\{x>\Gamma(\t),\ \t<\pi/2\}$, nor converge to $e_0$, the only possibility is that it intersects the line $\t=\pi/2$ at some finite point $(x_+,\pi/2),\ x_+>a/\p(0)$. By symmetry of the phase plane, $\gamma_+$ has $(0,\pi)$ as endpoint, i.e. it agrees with the orbit $\gamma_-$ described in Lemma \ref{orbitasmasmenos}, and is a compact arc joining $(0,0)$ and $(0,\pi)$. See Fig. \ref{fig:planofaseselipticas}, the orbit in blue. In conclusion, $\sig_+$ is a rotational, strictly convex $\Phi$-sphere.

Next, we prove the existence of a 1-parameter family of unduloid-type $\Phi$-surfaces. For that, fix some $x_0\in(0,a/\p(0))$ and let $\gamma_{x_0}(s)$ be the orbit in $\Theta_1$ passing through $(x_0,\pi/2)$ at $s=0$. For $s>0$, $\gamma_{x_0}(s)$ lies in the monotonicity region $\{x<\Gamma(\t),\ \t<\pi/2\}$ and it intersects the curve $\Gamma$, where the $\t$-coordinate reaches a minimum. Then, $\gamma_{x_0}$ lies in $\{x>\Gamma(\t),\ \t<\pi/2\}$, and since it cannot intersect $\gamma_+$ nor converge to $e_0$, has to reach the line $\t=\pi/2$ at some point $(\widehat{x}_0,\pi/2)$, with $a/\p(0)<\widehat{x}_0<x_+$. By symmetry, $\gamma_{x_0}$ is a closed orbit having $e_0$ in its inner region. This orbit generates a properly embedded (since $\t\in(0,\pi)$ and hence $z'(s)=\sin\t>0$) $\Phi$-surface, $\sig_{x_0}$, that is periodic and has an unduloid-type behavior. The parameter $x_0$ defining this family is the \emph{neck-size}, i.e. the minimum distance of $\sig_{x_0}$ to the axis of rotation.

\begin{figure}[h]
\centering
\includegraphics[width=.6\textwidth]{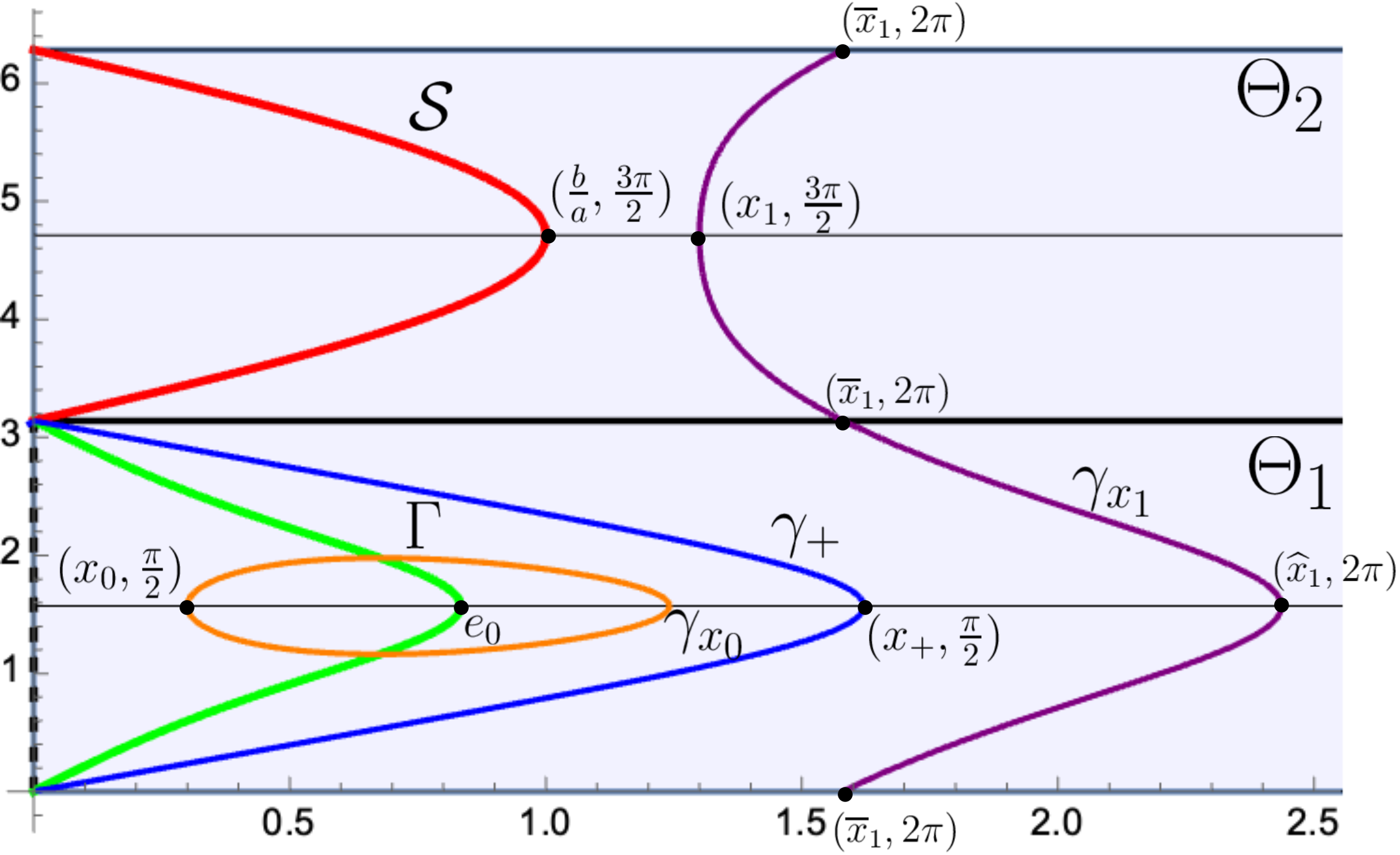}\hspace{1cm}\includegraphics[width=.2\textwidth]{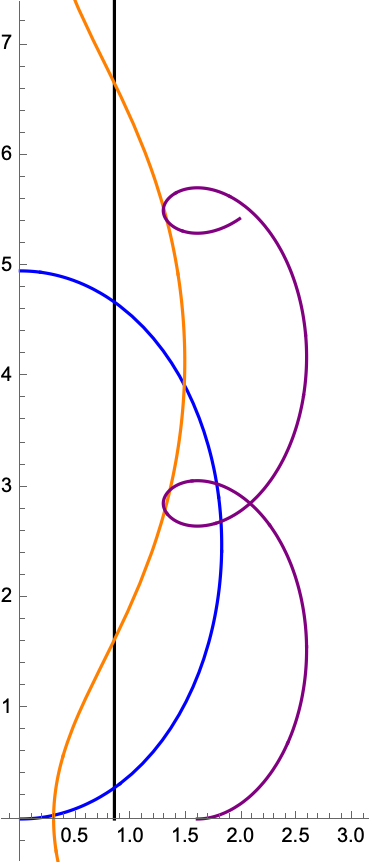}  
\caption{Left: the phase plane for $a>0,b>0$ and a prescribed function $\Phi$ under the hypotheses of Th. \ref{th:clasificacionespeciales}. The orbits correspond to the complete examples. Right: the profile curve of the $\Phi$-surface corresponding to each orbit.}
\end{figure}

Finally, we prove the existence of a 1-parameter family of nodoid-type $\Phi$-surfaces. Fix some $x_1>b/a$, and consider the orbit $\gamma_{x_1}(s)$ passing through $(x_1,3\pi/2)$ at $s=0$. From the monotonicity properties of $\Theta_2$, $\gamma_{x_1}(s)$ lies in $\{x>\s(\t),\ \t\in(3\pi/2,2\pi)\}$ for $s>0$ and has increasing both $x$ and $\t$-coordinates until reaching the line $\t=2\pi$ at some $(\overline{x}_1,2\pi)$ with $\overline{x}_1>x_1$. Similarly, and by symmetry, $\gamma_{x_1}(s)$ lies in $\{x>\s(\t),\ \t\in(\pi,3\pi/2)\}$ for $s<0$ and has $(\overline{x}_1,\pi)$ as endpoint.

At this point, when $s$ further decreases, $\gamma_{x_1}$ lies in the monotonicity region $\{x>\Gamma(\t),\ \t\in(\pi/2,\pi)\}$ of $\Theta_1$ until intersecting the line $\t=\pi/2$ at some $(\widehat{x}_1,\pi/2)$ with $\widehat{x}_1>x_+$. Again by symmetry, $\gamma_{x_1}$ has the point $(\overline{x}_1,0)$ as endpoint. Note that in particular, $\gamma_{x_1}$ does not intersect $\s$, hence the corresponding surface $\sig_{x_1}$ has no singularities. Moreover, it is properly immersed, periodic and has self-intersections, sharing the same properties as the constant mean curvature nodoids. The parameter $x_1$ defining this family is again the neck-size of each surface.

This completes the classification result in the complete case. Now we describe the non-complete examples.

First, fix some $\t_0\in(0,\pi/2]$. At the point $(0,\t_0)$, system \eqref{eq:planofases} has existence and uniqueness; here the fact that $b\neq0$ is paramount. Consequently, there exists an orbit $\gamma_{\t_0}$ such that $\gamma_{\t_0}(0)=(0,\t_0)$ and $\gamma_{\t_0}(s)$ lies in $\{x<\Gamma(\t),\ \t<\pi/2\}$ for $s>0$ small enough. By uniqueness, $\gamma_{\t_0}$ cannot intersect any $\gamma_{x_0}$ corresponding to the unduloids, nor $\gamma_+$ corresponding to the sphere. So, $\gamma_{\t_0}$ intersects the line $\t=\pi/2$ at some finite point $(x_{\t_0},\pi/2)$. By symmetry, $\gamma_{\t_0}$ is a compact arc having the point $(0,\pi-\t_0)$ as its other endpoint. See Fig. \ref{fig:elipticasnocompletas1} left, the orbit in orange.

\begin{figure}[h]
\centering
\includegraphics[width=.6\textwidth]{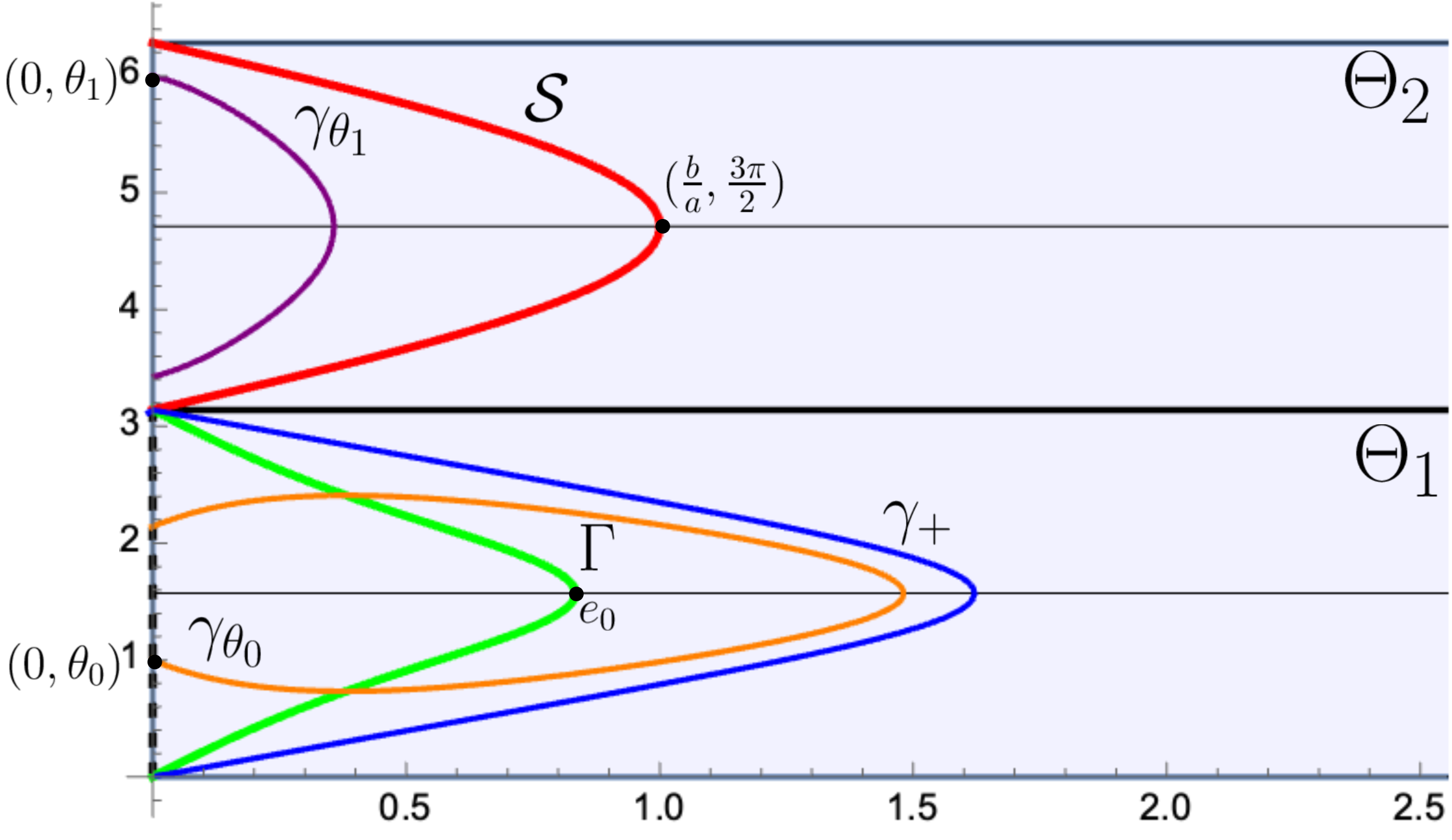}\hspace{1cm}\includegraphics[width=.15\textwidth]{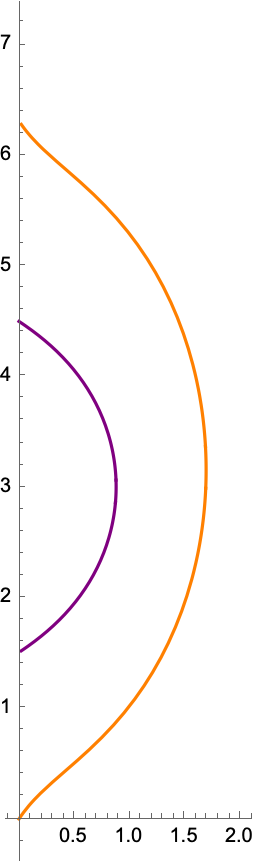}
\caption{Left: the orbits corresponding to the rotational $\Phi$-surfaces of type $\mathit{1.5,1.6}$ of Th. \ref{th:clasificacionespeciales}. Right: the profile curve of the $\Phi$-surface corresponding to each orbit.}
\label{fig:elipticasnocompletas1}
\end{figure}

This orbit generates a $\Phi$-surface, $\sig_{\t_0}$, that intersects the axis of rotation at a bottom point with angle function $\cos\t_0$ and at a top point with angle function $-\cos(\t_0)$. In particular, $\sig_{\t_0}$ is singular at these cusp points where it fails to be $C^1$, hence is non-complete. Recall that $\gamma_{\t_0}$ lies in both $\{x<\Gamma(\t)\}$ and $\{x>\Gamma(\t)\}$, hence the Gauss curvature of $\sig_{\t_0}$ changes of sign and vanishes at two parallels, which correspond to the points of $\gamma_{\t_0}$ intersecting $\Gamma$. In particular, none of these surfaces has strictly monotonous angle function. See Fig. \ref{fig:elipticasnocompletas1} right, the profile curve in orange. This proves the existence of surfaces of Item $\mathit{1.8}$ of Thm.\ref{th:clasificacionespeciales}.

Now, fix some $\t_1\in(3\pi/2,2\pi)$. Again, at the point $(0,\t_1)$ system \eqref{eq:planofases} has existence and uniqueness. Therefore, there exists an orbit $\gamma_{\t_1}$ such that $\gamma_{\t_1}(0)=(0,\t_1)$ and $\gamma_{\t_1}(s)$ lies in $\{x<\s(\t),\ \t\in(3\pi/2,2\pi)\}$ for $s>0$ small enough. We prove next that $\gamma_{\t_1}$ cannot converge to a point in $\s$.

\begin{claim}\label{claim}
The orbit $\gamma_{\t_1}$ cannot converge to a point in $\s$.
\end{claim}

\begin{proofclaim}
We begin by highlighting that a first integral of system \eqref{eq:planofases} for the particular choice $\phi=c\in\R$ is
\begin{equation}\label{integralprimera}
a(x(s)\sin\t(s)-x(s_0)\sin(\t_0))+\frac{b}{2}(\sin^2\t(s)-\sin^2\t(s_0))=\frac{c}{2}(x(s)^2-x(s_0)^2).
\end{equation}

In particular, if $c=0$ and we fix the initial condition $x(s_0)=0,\ \t(s_0)=\t_0\in(3\pi/2,2\pi)$ the first integral transforms into
\begin{equation}\label{integralprimeraparticular}
ax(s)\sin\t(s)+\frac{b}{2}(\sin^2\t(s)-\sin^2\t_0)=0.
\end{equation}

Let $\sigma_{\t_0}(t)=(x_{\t_0}(t),\t_{\t_0}(t))$\footnote{We use the parameter $t$ for the orbits of system \eqref{eq:planofases} for the particular choice $\phi=0$} be an orbit of system \eqref{eq:planofases} for $\phi=0$ and the initial condition $x(s_0)=0,\ \t(t_0)=\t_0\in(3\pi/2,2\pi)$. If $\sigma_{\t_0}(t)\rightarrow\s$ then $\t_{\t_0}\rightarrow\t_*$ and $x_{\t_0}\rightarrow-b\sin\t_*/a$, for some $\t_*\in(3\pi/2,\t_0)$. Substituting in \eqref{integralprimeraparticular} yields
$$
-\frac{b}{2}\sin^2\t_*-\frac{b}{2}\sin^2\t_0=0,
$$
which is a contradiction since $b>0$. Therefore, a simple substitution in \eqref{integralprimeraparticular} shows that $\sigma_{\t_0}(t)$ reaches the line $\t=3\pi/2$ at $(x_{\t_0},3\pi/2)$ with $x_{\t_0}=b/(2a)\cos^2\t_0$.

Let us consider now the orbit $\gamma_{\t_1}=(x_{\t_1}(s),\t_{\t_1}(s))$, and suppose that there exist two instants $s_1,t_1$ such that $\gamma_{\t_1}(s_1)=\sigma_{\t_0}(t_1)$. Comparing system \eqref{eq:planofases} for the prescribed function $\phi>0$ and the particular case $\phi=0$, respectively, yields
$$
\t_{\t_1}'(s_1)<\t_{\t_0}'(t_1).
$$

Now we stand in position to prove Claim \ref{claim}. Arguing by contradiction, suppose that $\gamma_{\t_1}$ converges to $\s$. Consider some $\t_0\in(\t_1,2\pi)$ and let $\sigma_{\t_0}$ be the orbit of system \eqref{eq:planofases} for $\phi=0$ and the initial condition $\sigma_{\t_0}(t_0)=(0,\t_0)$. We have already proven that this orbit has $(b/(2a)\cos^2\t_0,3\pi/2)$ as finite endpoint. Also, $\sigma_{\t_0}(t)$ for $t>t_0$ small enough lies above $\gamma_{\t_1}$; see Fig.\ref{fig:pruebaclaim}.

\begin{figure}[h]
\centering
\includegraphics[width=.6\textwidth]{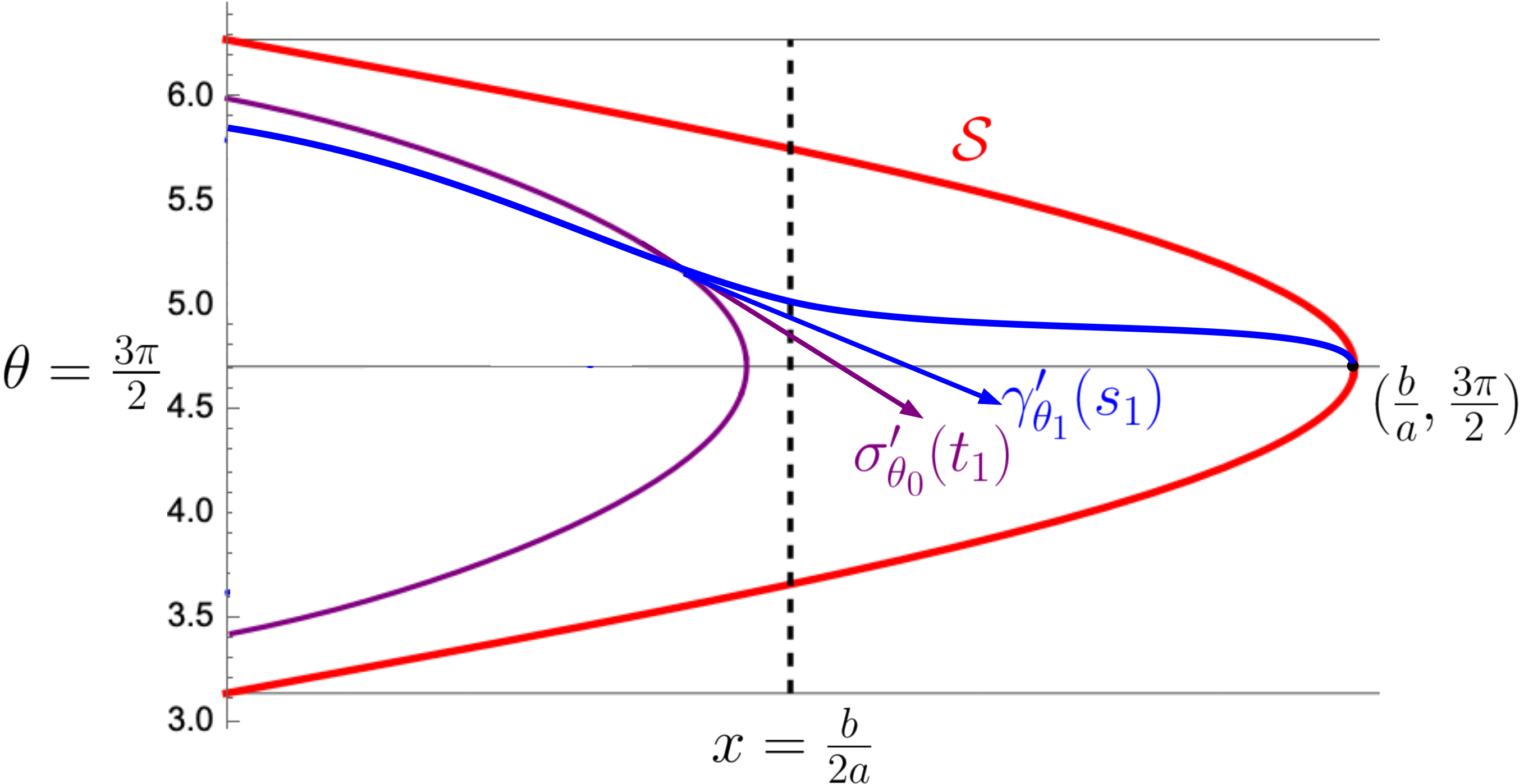}
\caption{The diagram of the contradiction in the proof of Claim \ref{claim}.}
\label{fig:pruebaclaim}
\end{figure}

If $\gamma_{\t_1}$ converges to $\s$, then by continuity $\gamma_{\t_1}(s_1)=\sigma_{\t_0}(t_1)$ with $s_1>s_0,\ t_1>t_0$. But at these instants we have $\t_{\t_1}'(s_1)>\t_{\t_0}'(t_1)$, arriving to a contradiction. Thus, $\gamma_{\t_1}$ intersects the line $\t=3\pi/2$ at some finite point $(x_{\t_1},3\pi/2)$, with $x_{\t_1}<x_{\t_0}<b/(2a)$.
\end{proofclaim}

Consequently, $\gamma_{\t_1}$ passes through some $(x_{\t_1},3\pi/2),\ x_{\t_1}<b/(2a)$, and by symmetry $\gamma_{\t_1}$ closes at the point $(0,3\pi/2-\t_1)$; see Fig. \ref{fig:elipticasnocompletas1} left, the orbit in purple. The $\Phi$-surface $\sig_{\t_1}$ generated by $\gamma_{\t_1}$, is compact and intersects the axis of rotation at two singular cusp points, where $\sig_{\t_1}$ fails to be $C^1$. Moreover, $\sig_{\t_1}$ has strictly positive Gauss curvature. See Fig. \ref{fig:elipticasnocompletas1} right, the profile curve in purple. This proves the existence of surfaces of Item $\mathit{1.5}$ of Thm.\ref{th:clasificacionespeciales}.

We prove the existence of the last non-complete examples. First, note that the definition of $\widehat{x}_1$ in the construction of the complete, nodoidal $\Phi$-surfaces, maps the interval $(b/a,\infty)$ to $(x_1^\infty,\infty)$, for some $x_1^\infty\geq x_+$\footnote{In fact, $x_1^\infty=\lim \widehat{x}_1$ as the corresponding $x_1\rightarrow b/a$.}. We show that $x_1^\infty\neq x_+$.

Arguing by contradiction, suppose that $x_1^\infty=x_+$ and consider the orbit $\gamma$ passing through $(\overline{x},\pi)$, with $0<\overline{x}<b/a$. Then, for $s<0$, $\gamma(s)$ lies in $\{x>\Gamma(\t),\ \t\in(\pi/2,\pi)\}$ and stays there for $s$ decreasing until converging to the line $\t=\pi/2$. However, $\gamma$ cannot reach any $(\widehat{x},\pi/2)$ since necessarily $\widehat{x}\geq x_+$ and we already described the orbits passing through these points as the complete nodoids and the sphere. In particular, by the way it was defined, the orbit $\gamma_{x_1^\infty}$ passing through $(x_1^\infty,\pi/2)$ has to converge to the point $(b/a,3\pi/2)$ after intersecting the line $\t=\pi$ at some $(\overline{x}_1^\infty,\pi)$. By symmetry, there exists an orbit in $\{x>\s(\t),\t\in(3\pi/2,2\pi)\}$ having $(\overline{x}_1^\infty,2\pi)$ as endpoint and converging to the point $(b/a,3\pi/2)$. By analogy, we name this orbit again as $\gamma_{x_1^\infty}$; see Fig. \ref{fig:nodnocompletos} left, the orbit in purple.

The $\Phi$-surface corresponding to the orbit $\gamma_{x_1^\infty}$ has a certain nodoidal shape but with a parallel removed consisting of singular points, hence is non-complete, and this behavior is repeated after a vertical translation. See Fig. \ref{fig:nodnocompletos} right, the profile curve in purple.

\begin{figure}[h]
\centering
\includegraphics[width=.6\textwidth]{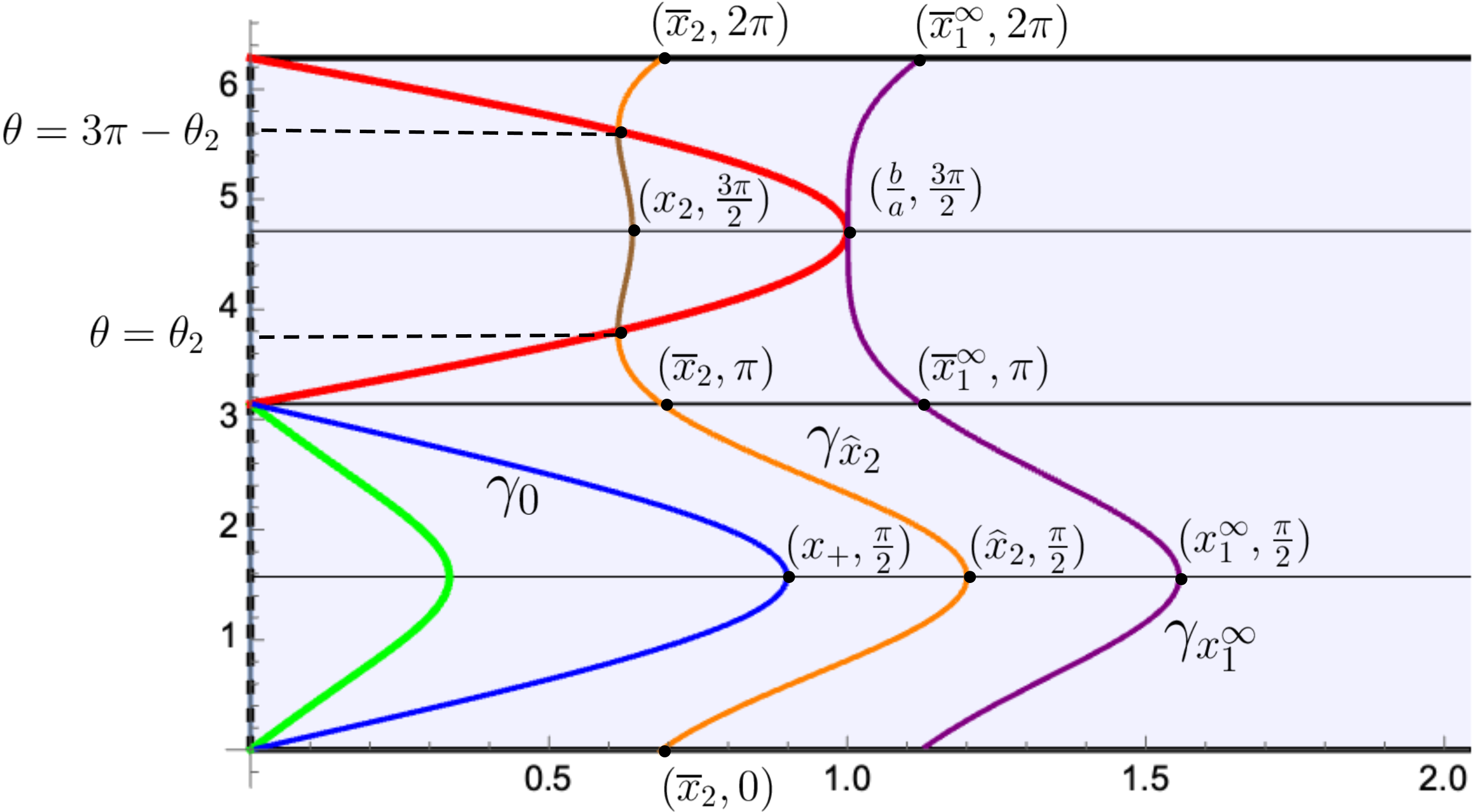}\hspace{.5cm}\includegraphics[width=.35\textwidth]{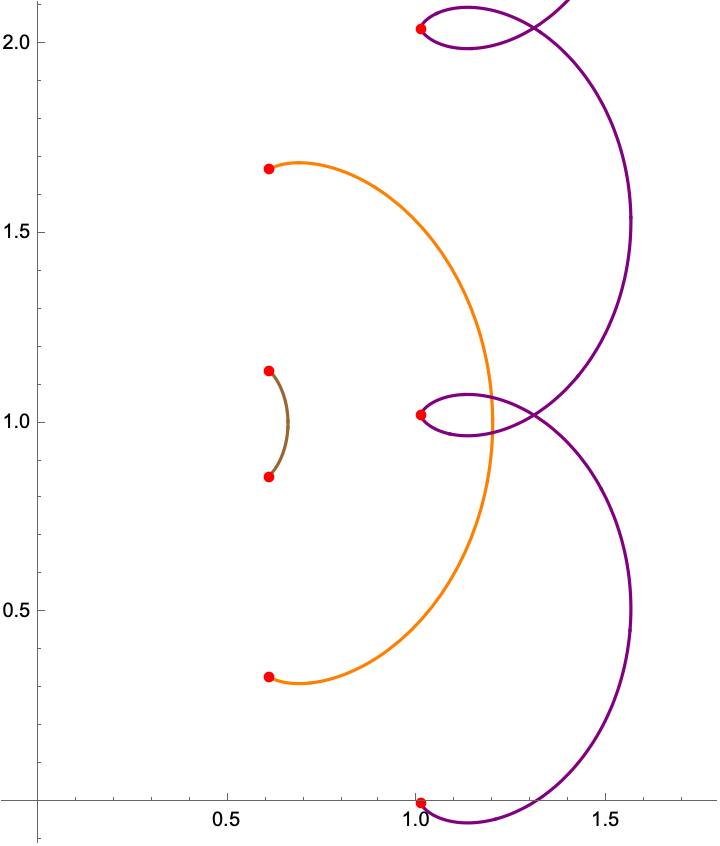}
\caption{Left: the orbits corresponding to the rotational $\Phi$-surfaces of type $\mathit{1.7,1.8}$ of Th. \ref{th:clasificacionespeciales}. Right: the profile curve of the $\Phi$-surface corresponding to each orbit.}
\label{fig:nodnocompletos}
\end{figure}

Now fix some $\widehat{x}_2\in(x_+,x_1^\infty)$ and consider $\gamma_{\widehat{x}_2}(s)$ the orbit passing through $(\widehat{x}_2,\pi/2)$ at $s=0$. For $s>0$, $\gamma_{\widehat{x}_2}$ intersects the line $\t=\pi$ at some $(\overline{x}_2,\pi),\ \overline{x}_2<\widehat{x}_2$, and by symmetry it has $(\overline{x}_2,0)$ as endpoint for $s<0$. When $s>0$ further increases, $\gamma_{\widehat{x}_2}$ enters the region $\{x>\s(\t),\ \t\in(\pi,3\pi/2)\}$ and since $\gamma_{\widehat{x}_2}$ cannot intersect any other orbit, it ends up converging to some $(\s(\t_2),\t_2)$ with $\t_2\in(\pi,3\pi/2)$. By symmetry, there is a symmetric orbit, which we will name again $\gamma_{\widehat{x}_2}$ by analogy, that has the point $(\s(\t_2),3\pi-\t_2)$ as endpoint and converges to $(\overline{x}_2,2\pi)$ as $s$ increases. See Fig. \ref{fig:nodnocompletos} left, the orbit in orange. 

The corresponding $\Phi$-surface has strictly positive Gauss curvature, is homeomorphic to an annulus and has non-monotonous height function. See Fig. \ref{fig:nodnocompletos} right, the profile curve in orange. This proves the existence of the surfaces of Item $\mathit{1.7}$ in Thm.\ref{th:clasificacionespeciales}.

Finally, since the orbits of the phase plane foliate it, there exists an orbit in the region $\{x<\s(\t)\}$ whose extremes converge to the points $(\s(\t_2),\t_2)$ and $(\s(\t_2),3\pi-\t_2)$, and that intersects the line $\t=3\pi/2$ at some $(x_2,3\pi/2)$. See Fig. \ref{fig:nodnocompletos} left, the orbit in brown. The corresponding $\Phi$-surface has all the properties stated in Item $\mathit{1.6}$ of Thm.\ref{th:clasificacionespeciales}, concluding its proof.
\end{proof}

\subsection{The case $a>0,b<0$.}\label{sec32} The classification result obtained in this section is the following:

\begin{teo}\label{th:clasificacionelipticas2}
Let be $a>0,b<0$ and $\Phi\in C^1(\S^2)$ positive and even. 

A complete, rotational $\Phi$-surface of elliptic type is one of the following:
\begin{enumerate}
\item[1.1] A vertical, circular cylinder;
\item[1.2] a strictly convex sphere;
\item[1.3] a properly embedded, periodic surface of unduloid-type; and
\item[1.4] a properly immersed, periodic surface of nodoid-type, with self-intersections.
\end{enumerate}
Besides open pieces of the above examples, a non-complete, rotational $\Phi$-surface of elliptic type is one of the following:
\begin{enumerate}
\item[1.5] a surface of $K>0$, with two cusp points at the $z$-axis and strictly monotonous height;
\item[1.6] a surface of $K>0$, homeomorphic to an annulus and strictly monotonous height;
\item[1.7] a surface of $K>0$, with two cusp points at the $z$-axis and non-monotonous height; and
\item[1.8] a surface of $K$ changing sign, homeomorphic to an annulus and strictly monotonous height.
\end{enumerate}
\end{teo}

\begin{proof}
We define $\phi$ in terms of $\Phi$ by Eqn. \eqref{condicionrotsim} as usual. In this case, the structure of the phase plane is as shown in Fig. \ref{fig:orbitascompletascaso2}: both curves $\s$ and $\Gamma$ lie entirely in $\Theta_1$ and are compact arcs joining the points $(0,0)$ and $(0,\pi)$. 

\begin{figure}[h]
\centering
\includegraphics[width=.45\textwidth]{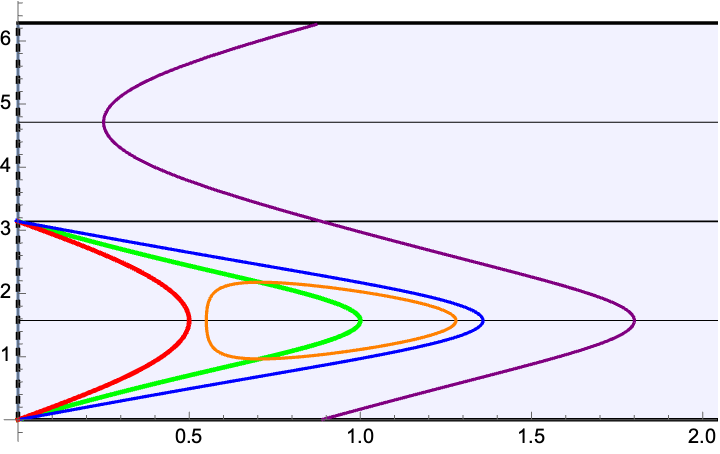}
\caption{The structure of the phase plane in the elliptic case for $a>0,b<0$, and the orbits corresponding to the complete rotational $\Phi$-surfaces.}
\label{fig:orbitascompletascaso2}
\end{figure}

The existence of the complete examples carries over verbatim as in the case $a>0,b>0$ exhibited in Thm.\ref{th:clasificacionespeciales}, and the orbits of these complete examples are shown in Fig. \ref{fig:orbitascompletascaso2}. The description of the surfaces of type $\mathit{1.5}$ and $\mathit{1.6}$ is also the same as in the case $a>0,b>0$; see Fig. \ref{fig:orbitasnocompletascaso2}, the orbits and profile curves in black and brown. Details are skipped at this point.

\begin{figure}[h]
\centering
\includegraphics[width=.5\textwidth]{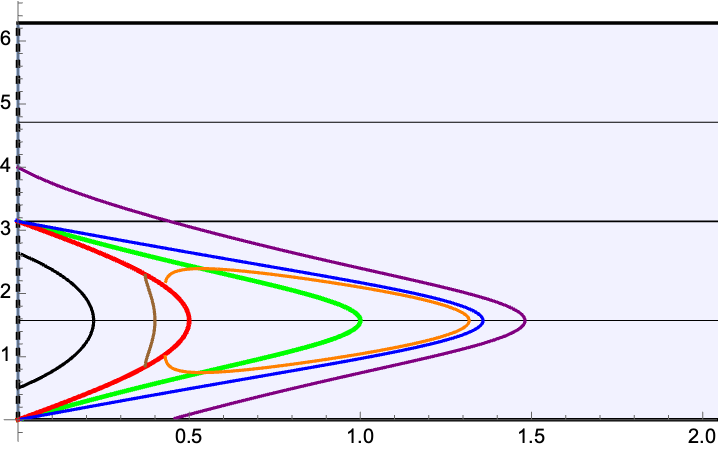}\hspace{1cm}\includegraphics[width=.17\textwidth]{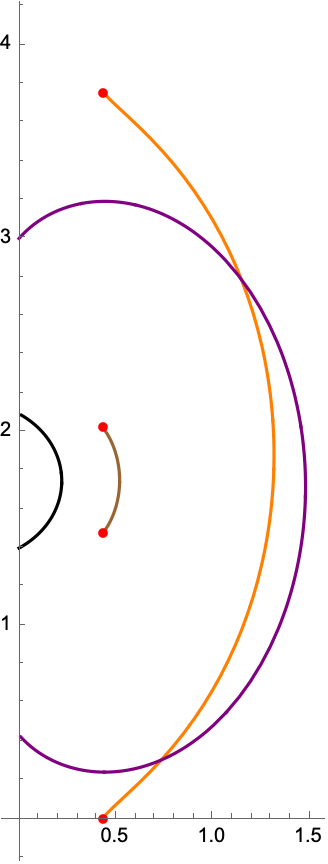}
\caption{Left: the orbits corresponding to the rotational $\Phi$-surfaces of type $\mathit{1.5-1.8}$ of Th. \ref{th:clasificacionelipticas2}. Right: the profile curve of the $\Phi$-surface corresponding to each orbit.}
\label{fig:orbitasnocompletascaso2}
\end{figure}

For the existence of the examples of type $\mathit{1.7}$, fix some $\t_0\in(\pi,3\pi/2]$ and let $\gamma_{\t_0}$ be the orbit having $(0,\t_0)$ as endpoint at the instant $s=0$. By symmetry of $\Theta_2$, $\gamma_{\t_0}$ has the point $3\pi-\t_0$ as endpoint, and ends up at some $(x_1,2\pi)$. For $s<0$, $\gamma_{\t_0}$ intersects the line $x=\pi$ at $(x_1,\pi)$ and then enters the region $\Theta_1$. Since $\gamma_{\t_0}$ cannot intersect the orbit corresponding to the sphere, nor any orbit of the nodoids, $\gamma_{\t_0}$ intersects the line $x=\pi/2$. By symmetry, $\gamma_{\t_0}$ intersects the line $x=0$ at $(x_1,0)$. See Fig \ref{fig:orbitasnocompletascaso2}, the orbit and profile curve in purple.

The $\Phi$-surface generated by this orbit intersects the axis of rotation at two cusp points, its height function is non-monotonous, and its Gauss curvature is everywhere positive. This proves the existence of the rotational examples of type $\mathit{1.7}$.

Finally, we prove the existence of the examples of type $\mathit{1.8}$. Name $\mathcal{W}_0$ to the inner region bounded by the orbit $\gamma_0$ of the rotational $\Phi$-sphere and the curve $\s$ of singular points. Fix some $(x_0,\t_0)\in\mathcal{W}_0-\Gamma$ and let $\gamma$ be the orbit passing through such point. Assume moreover that $\gamma$ does not correspond to an unduloid. It is clear that the only possibility for $\gamma$ is to be a bi-graph over the line $x=\pi/2$, intersecting it at a single point, having two endpoints located at the curve $\s$ and intersecting the curve $\Gamma$ twice. See Fig \ref{fig:orbitasnocompletascaso2}, the orbit and profile curve in orange.

The $\Phi$-surface generated by this orbit has the topology of an annulus, has strictly increasing height function (since $\gamma$ is strictly contained in $\Theta_1$) and its Gauss curvature changes sign. This exhibits the existence of the rotational examples of type $\mathit{1.8}$ and proves Thm.\ref{th:clasificacionelipticas2}.
\end{proof}

%\begin{obs}
%Almost all the examples described in Thms. \ref{th:clasificacionespeciales} and \ref{th:clasificacionelipticas2} already appeared in \cite{RoSa,SaTo1}. Nonetheless, the following were missing:
%\begin{itemize}
%\item In Section \textbf{IV} in \cite{RoSa}, the examples of type $\mathit{1.8}$ of Thm.\ref{th:clasificacionespeciales} only appeared when the intersection with the rotation axis is tangential. In the case that the intersection is transverse, only the examples of type $\mathit{1.5}$ of Thms. \ref{th:clasificacionespeciales} and \ref{th:clasificacionelipticas2} were described, missing the ones of type $\mathit{1.8}$ with transverse intersection.
%\item In Thms. 3, 4 in \cite{SaTo1}, the rotational non-complete examples were classified. The authors deduced in Thm.3 the existence of the examples of type $\mathit{1.8}$ of Thm.\ref{th:clasificacionelipticas2}. In Thm.4, (1), they exhibited the existence of examples $\mathit{1.5}$ of Thms. \ref{th:clasificacionespeciales} and \ref{th:clasificacionelipticas2}; examples $\mathit{•}.$
%\end{itemize}
%\end{obs}

\begin{obs}
Most of the surfaces obtained in Thms. \ref{th:clasificacionespeciales} and \ref{th:clasificacionelipticas2} already appeared in \cite{RoSa,SaTo1}; for example, the complete Delaunay-type ones. Regarding the non-complete:
\begin{itemize}
\item Surfaces of type $\mathit{1.5,1.6}$ were described in Section IV in \cite{RoSa} and Thm.4, (1) and (3), in \cite{SaTo1}.
\item Surfaces of type $\mathit{1.8}$ of Thm.\ref{th:clasificacionespeciales} were described in Section IV in \cite{RoSa}, but only when the intersection with the rotation axis is tangential. For a transverse intersection, these surfaces were missing.
\item Surfaces of type $\mathit{1.8}$ of Thm.\ref{th:clasificacionelipticas2} were described in Thm.3 in \cite{SaTo1}.
\item Surfaces of type $\mathit{1.7}$ of Thms. \ref{th:clasificacionespeciales} and \ref{th:clasificacionelipticas2} do not appear in neither \cite{RoSa} nor \cite{SaTo1}. In \cite{SaTo1}, (1), it seems to appear a portion of these surfaces, the one being a concave graph onto the rotation axis, but none is said about if it intersects the rotation axis (type $\mathit{1.7}$ of Thm.\ref{th:clasificacionelipticas2}) or if it stays at a positive distance of it (type $\mathit{1.7}$ of Thm.\ref{th:clasificacionespeciales}).
\end{itemize}
Therefore, Thms. \ref{th:clasificacionespeciales} and \ref{th:clasificacionelipticas2} provide a complete classification of the rotational linear Weingarten surfaces of elliptic type, with an explicit behavior of the profile curve of each of their examples.
\end{obs}

\section{The hyperbolic case}\label{sec4}
In this section we classify rotational $\Phi$-surfaces of hyperbolic type in the sense of Def. \ref{tipoEDP}, i.e. the constants $a,b$ and the prescribed function $\Phi\in C^1(\S^2)$ related by Eqn. \eqref{eq:defiPLW} satisfy $a^2+b\Phi<0$. Note that necessarily both $b,\Phi\neq 0$, hence we may assume $b=1$ and Eqn. \eqref{eq:defiPLW} is
$$
2aH+K=\Phi.
$$
Therefore, the surface is of hyperbolic type if and only if $\Phi<-a^2$.

Our classification result covers the particular but important case when $\Phi=c\in\R$, corresponding to hyperbolic linear Weingarten surfaces. Although the classification of these rotational surfaces was exhibited by López in \cite{Lop}, there exists a mistake. He described all the rotational linear Weingarten surfaces of hyperbolic type in terms of a parameter $x_0>0$, $x_0\neq 1/a$, for which he claimed that when $x_0>1/a$, all the corresponding surfaces are complete. This is actually false, as we will explain in the discussions of Item $\mathit{4}$ of the following result; see Obs. \ref{observation} for further details.

\begin{teo}\label{th:clasificacionhiperbolicas}
Let be $a\in\R$, $\Phi\in C^1(\S^2)$, $\Phi$ even, such that $\Phi<-a^2$, and $\phi$ the 1-dimensional function defined by \eqref{condicionrotsim} in terms of $\Phi$. Then, rotational $\Phi$-surfaces are described by a parameter $x_0>0,\ x_0\neq 1/a$, such that:
\begin{enumerate}
\item If $x_0<-a/\p(0)$, the surface is non-complete, of positive Gauss curvature and intersects the rotation axis at two cusp points;
\item if $x_0=-a/\p(0)$, the surface is the vertical right cylinder of radius $x_0$;
\item if $x_0\in(-a/\p(0),1/a)$, the surface is non-complete and of negative Gauss curvature; and
\item if $x_0>1/a$, depending on the prescribed function, two possibilities may occur:
\begin{itemize}
\item[4.1.] If $\phi\leq-2a^2$, for every $x_0>1/a$, the surface is complete and periodic, with a nodoidal behavior; or
\item[4.2.] if $\phi>-2a^2$, there exists $x_\infty>1/a$ such that:
\begin{itemize}
\item[4.2.1.] if $x_0\leq x_\infty$, then the surface is homemorphic to a sphere with two cusp points at the axis of rotation, and in particular is non-complete; and
\item[4.2.2.] if $x_0>x_\infty$, then the surface is complete and periodic, with a nodoidal behavior.
\end{itemize}
\end{itemize}
\end{enumerate}
\end{teo}

\begin{proof}
As usual, let $\phi$ be the 1-dimensional function defined by Eqn. \eqref{condicionrotsim} in terms of $\Phi$. After a change of the orientation if necessary, we assume $a>0$. Since $a>0$ both curves $\Gamma$ and $\s$ only exist for $\t\in[\pi,2\pi]$, i.e. they lie in $\Theta_2$. Also, the condition $\p<-a^2$ yields $\Gamma(\t)<\s(\t)$ for every $\t\in(\pi,2\pi)$.

First, take some $x_0\in(0,-a/\p(0))$ and consider $\gamma_{x_0}(s)$ the orbit passing through $(x_0,3\pi/2)$ at $s=0$. For $s>0$, $\g0$ lies in $\{x<\Gamma(\t),\ \t\in(\pi,3\pi/2)\}$ and stays there until converging to some $(0,\t_0)$ with $\t_0\in(\pi,3\pi/2)$. Since $\p$ is even, $\g0$ is symmetric with respect to the line $\t=3\pi/2$ and in particular has the point $(0,3\pi-\t_0)$ as endpoint. Therefore, the corresponding $\Phi$-surface has two cusp points at the axis of rotation and hence is non-complete. Moreover, since $\t'<0$ and $\sin\t<0$, from Eqn. \eqref{eq:curvprin} we conclude that its Gauss curvature is positive everywhere.

\begin{figure}[h]
\centering
\hspace{-.5cm}\includegraphics[width=.6\textwidth]{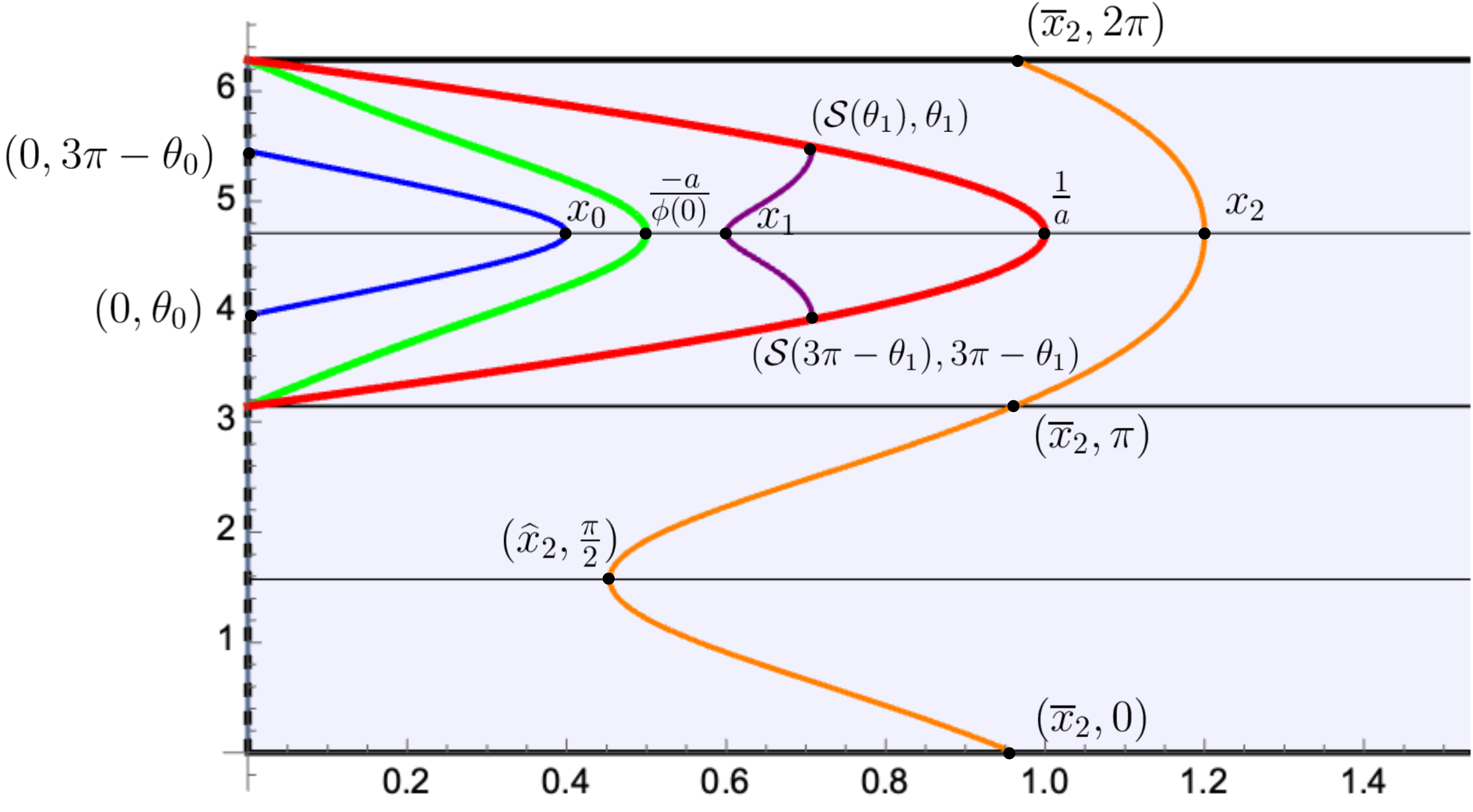}\hspace{1cm}\includegraphics[width=.25\textwidth]{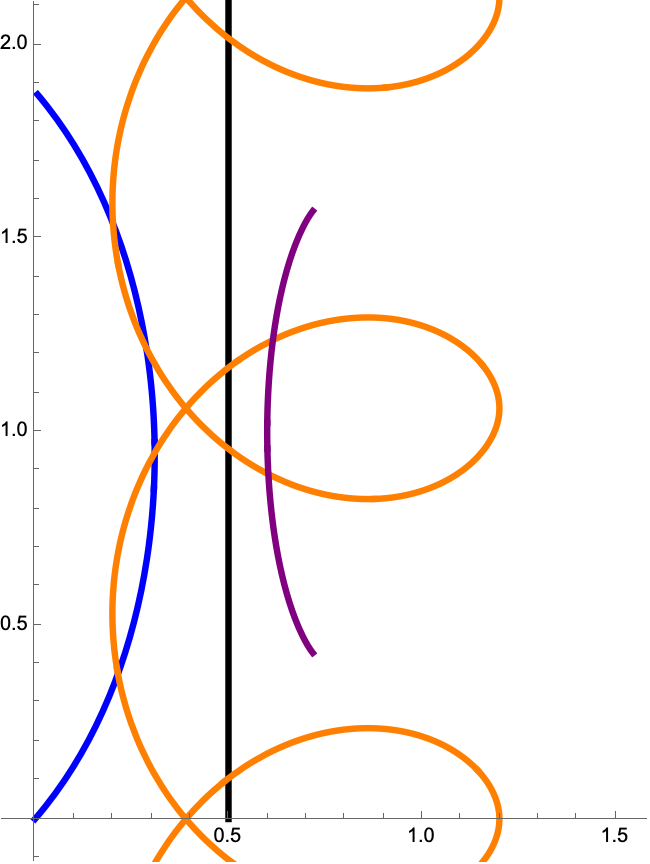}
\caption{Left: The phase plane for $b=1$ and $\Phi<-a^2$, and the orbits appearing in the statement of Th. \ref{th:clasificacionhiperbolicas}. Right: the profile curve corresponding to such orbits.}
\label{fig:hiperbolicas}
\end{figure}

Now consider the equilibrium point $e_0=(-a/\p(0),3\pi/2)$. This equilibrium generates a vertical, right circular cylinder of radius $-a/\p(0)$ and constant mean curvature $\p(0)/(2a)$\footnote{Recall that this cylinder is parametrized with strictly decreasing height function, hence the unit normal induced by this parametrization is the outwards one and the mean curvature is negative.}.

Next, fix some $x_1\in(-a/\p(0),1/a)$ and consider the orbit $\gamma_{x_1}$ passing through $(x_1,3\pi/2)$ at $s=0$. For $s>0$, $\gamma_{x_1}$ lies in the region $\{x\in(\Gamma(\t),\s(\t)),\ \t\in(3\pi/2,2\pi)\}$ and stays there until converging to some $(\s(\t_1),\t_1)$. By symmetry, $\gamma_{x_1}$ also converges to $(\s(3\pi-\t_1),3\pi-\t_1)$ for $s<0$. In conclusion, the corresponding $\Phi$-surface is non-complete as it converges to two circles of singular points. Moreover, this time $\t'>0$ and $\sin\t<0$, and so by \eqref{eq:curvprin} we conclude that its Gauss curvature is negative everywhere.

Note that no orbit can pass through the point $(1/a,3\pi/2)$, since it correspond to a singularity of Eqn. \eqref{eq:planofases}, but it may converge to it in the same fashion as the orbits $\gamma_{x_1}$ converge to points in $\s$. 

In order to finally describe the orbits passing through some $(x_0,3\pi/2)$, with $x_0>1/a$, we study deeply the particular case $\Phi=c\in\R,\ c<-a^2$. Consider the orbit 
\begin{equation}\label{orbitac}
\gamma_c(s)=(x_c(s),\t_c(s)),\hspace{.5cm} x_c(0)=0,\ \t_c(0)=\pi/2.
\end{equation}
In virtue of Eqn. \eqref{integralprimera}, a first integral is given by
\begin{equation}\label{integralprimerahiperbolica}
ax(s)\sin\t(s)+\frac{1}{2}(\sin^2\t(s)-1)=\frac{c}{2}x(s)^2.
\end{equation}
We study three different possibilities for the value $c$.

\underline{\textbf{Case $c=-2a^2$.}} 

Let $\gamma_{-2a^2}$ be the orbit given by \eqref{orbitac} for $c=-2a^2$. By monotonicity, $\gamma_{-2a^2}$ intersects the line $\t=\pi$ at some $(\overline{x},\pi)$ and then either converges to $\s$ or intersects the line $\t=3\pi/2$. We discuss both cases.
\begin{itemize}
\item If $\gamma_{-2a^2}$ converges to $\t=3\pi/2$, then $\t_{-2a^2}(s)\rightarrow 3\pi/2$ and $x_{-2a^2}(s)\rightarrow 1/a$ by just substituting in \eqref{integralprimerahiperbolica}. In particular, $\gamma_{-2a^2}$ never intersects $\t=3\pi/2$. 
\item If $\gamma_{-2a^2}$ converges to $\s$, then $\t_{-2a^2}(s)\rightarrow\t_0$ and $x_{-2a^2}(s)\rightarrow\s(\t_0)$, for some $\t_0\in(\pi,3\pi/2]$. Since $\s(\t_0)=-1/a\sin\t_0$, a straightforward substitution in \eqref{integralprimerahiperbolica} yields $\t_0=3\pi/2$ and so $x_{-2a^2}(s)\rightarrow 1/a$, that is $\gamma_{-2a^2}$ converges to the point $(1/a,3\pi/2)$.
\end{itemize}
Since by monotonicity $\gamma_{-2a^2}$ must either converge to $\s$ or $\t=3\pi/2$, we conclude that $\gamma_{-2a^2}$ actually converges to $(1/a,3\pi/2)$, without intersecting it. By symmetry of $\Theta_1$, the orbit $\gamma_{-2a^2}$ ends up at the point $(\overline{x},0)$. Also, by symmetry of $\Theta_2$, there exists an orbit, called again $\gamma_{-2a^2}$ in analogy, with one end converging to $(1/a,3\pi/2)$ and the other end reaching the point $(\overline{x},2\pi)$. See Fig. \ref{fig:linealesW}, the orbit in orange. In particular, the $\Phi$-surface generated by this orbit is non-complete, since it converges to a parallel of singular points and intersects the axis of rotation at cusp points.

\begin{figure}[h]
\centering
\includegraphics[width=.5\textwidth]{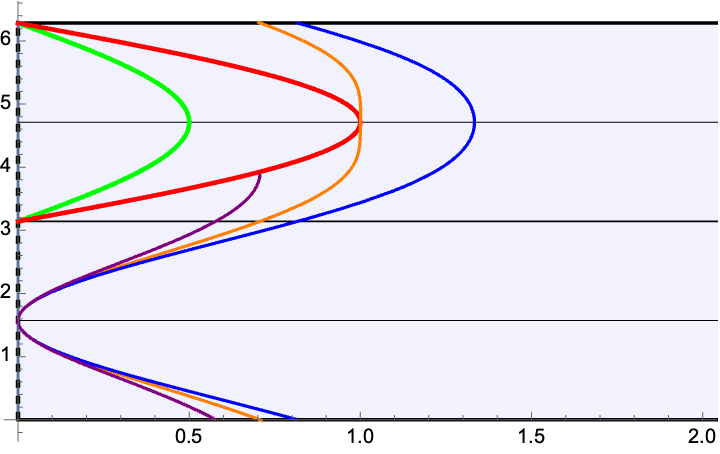}
\caption{The phase plane for the case $b=1,\ \Phi=c\in\R,\ c<-a^2$, and the distinct orbits depending on the sign of $c+2a^2$.}
\label{fig:linealesW}
\end{figure}

Take some $x_0>1/a$ and consider $\g0$ the orbit passing through $(x_0,3\pi/2)$. The uniqueness of \eqref{eq:planofases} yields that $\g0$ and $\gamma_{-2a^2}$ cannot intersect, hence $\g0$ intersects the line $\t=\pi/2$ at some $(\widehat{x_0},\pi/2)$ with $\widehat{x_0}>0$. In particular, all the $\Phi$-surfaces generated by these orbits are complete, since none of them have singular points nor intersect the axis of rotation at cusp points. Moreover, they have a nodoidal behavior.

Now take some $\t_0\in(\pi/2,\pi)$ and consider $\gamma_{\t_0}$ the orbit passing through $(0,\t_0)$. Again by uniqueness $\gamma_{\t_0}$ cannot intersect $\gamma_{-2a^2}$, and so $\gamma_{\t_0}$ must converge to a point in $\s$. In particular, none of the $\Phi$-surfaces generated by these orbits is complete, since all of them converge to singular points and intersect the axis of rotation at cusp points.

%We remark that these orbits $\gamma_{\t_0}$ are in some sense the \emph{continuation} of the orbits $\gamma_{x_1}$ previously described, when both orbits converge to the same point in $\s$.

\underline{\textbf{Case $c<-2a^2$.}}

Let $\gamma_c$ be the orbit given by \eqref{orbitac}. Suppose that $\gamma_c$ has some $(\overline{x},3\pi/2),\ \overline{x}>1/a,$ as finite endpoint. Substituting in \eqref{orbitac} we conclude that $\overline{x}=-2a/c$. Now, the condition $c<-2a^2$ ensures us that $-2a/c<1/a$, which contradicts the fact that $\overline{x}>1/a$. Hence, $\gamma_c$ must converge to $\s$. See Fig. \ref{fig:linealesW}, the orbit in purple. 

Take some $x_0>1/a$ and let $\g0(s)$ be the orbit passing through $(x_0,3\pi/2)$ at $s=0$. Since it cannot intersect $\gamma_c$, for $s>0$ it must reach the line $\t=\pi/2$ at some $(\widehat{x_0},\pi/2)$ with $\widehat{x_0}>0$. By the same discussion made for the case $c=-2a^2$, the $\Phi$-surface corresponding to $\g0$ is complete and has a nodoidal behavior.

Now take some $\t_0\in(\pi/2,\pi)$ and consider $\gamma_{\t_0}$ the orbit passing through $(0,\t_0)$. This orbit must end up converging to $\s$ without intersecting $\gamma_c$, and the corresponding $\Phi$-surface is non-complete.

\underline{\textbf{Case $c>-2a^2$.}}

Finally, let $\gamma_c$ be the orbit given by \eqref{orbitac} and assume $\gamma_c\rightarrow\s$. Hence, $\t_c\rightarrow\t_*$ and $x_c\rightarrow-1/a\sin\t_*$. Substituting in \eqref{integralprimerahiperbolica} we obtain
$$
-\frac{1}{2}\sin^2\t_*-\frac{1}{2}=\frac{c}{2a^2}\sin^2\t_*.
$$
Now, the condition $c>-2a^2$ yields
$$
-\frac{1}{2}\sin^2\t_*-\frac{1}{2}>-\sin^2\t_*,
$$
and so $1<\sin^2\t_*$, which is obviously a contradiction. Therefore, $\gamma_c$ must reach the line $\t=3\pi/2$ at $(x_c,3\pi/2)$, with $x_c=-2a/c>1/a$. See Fig. \ref{fig:linealesW}, the orbit in blue.

Now take some $x_0>1/a$ and let $\g0$ be the orbit passing through $(x_0,3\pi/2)$.
\begin{itemize}
\item If $x_0\leq x_c$, $\g0$ reaches the line $x=0$ at some $(0,\t_{x_0})$. The corresponding $\Phi$-surface is non-complete, since it intersects the axis of rotation at cusp points.
\item If $x_0>x_c$, then $\g0$ reaches the line $\t=\pi/2$ at some $\widehat{x_0}$ with $\widehat{x_0}>0$. By the same discussions made for the cases $c\leq-2a^2$, the corresponding $\Phi$-surface is complete and has a nodoidal behavior.
\end{itemize}
\hspace{.5cm}

Once we have described the orbits of the phase plane for the particular case when $\Phi=c\in\R$, we go back to the case where $\Phi<-a^2$, is an arbitrary function $\Phi\in C^1(\S^2)$. 
\begin{enumerate}
\item Suppose that $\Phi\leq-2a^2$. Let be $x_0>1/a$ and $\g0(s)=(x_{x_0}(s),\t_{x_0}(s))$ the orbit of \eqref{eq:planofases} for $\Phi$ passing through $(x_0,3\pi/2)$. We take $c=-2a^2$ and $\gamma_{-2a^2}(t)$ the orbit of \eqref{eq:planofases} for $\Phi=-2a^2$ passing through $(0,\pi/2)$, which we already proved that converges to the point $(1/a,3\pi/2)$. If $\g0(s_0)=\gamma_{-2a^2}(t_0)$, the same comparison argument used in the proof of Claim \ref{claim} yields
$$
\t_{x_0}'(s_0)\leq\t_{-2a^2}(t_0).
$$

First, assume that $\g0$ ends up intersecting the line $x=0$ at some $(0,\t_0)$ with $t_0\in(\pi/2,\pi)$. Then, by continuity $\g0$ and $\gamma_{-2a^2}$ would intersect transversely at a point where $\t_{x_0}'(s_0)<\t_{-2a^2}(t_0)$, which is impossible. Second, assume that $\g0$ ends up intersecting the point $(0,\pi/2)$. We take $x_1\in(1/a,x_0)$ and $\gamma_{x_1}$ the orbit passing through $(x_1,3\pi/2)$. This orbit cannot end up at $(0,\pi/2)$ by uniqueness, hence must intersect $x=0$ at some $(0,\t_1)$ with $\t_1\in(0,\pi)$. Again, $\gamma_{x_1}$ would intersect transversely $\gamma_{-2a^2}$, a contradiction.

Therefore, $\g0$ ends up intersecting the line $\t=\pi/2$ at some $(\widehat{x_0},\pi/2)$, hence $\g0$ generates a complete $\Phi$-surface that has no singularities and a nodoidal behavior.
\item Suppose that $\Phi>-2a^2$. Fix some $c$ such that $\Phi>c>-2a^2$ and let $\gamma_c(t)=(x_c(t),\t_c(t))$ be the orbit of \eqref{eq:planofases} for $\Phi=c$ passing through $(0,\pi/2)$, which we already proved that intersects $\t=3\pi/2$ at some $(x_c,3\pi/2),\ x_c>1/a$.

Let be $x_0>1/a$ and $\g0(s)=(x_{x_0}(s),\t{x_0}(s))$ the orbit passing through $(x_0,3\pi/2)$. This time, the comparison between the functions $\t'_{x_0}$ and $\t'_c$ at an intersection point between $\g0$ and $\gamma_c$ yields
$$
\t'_{x_0}(s_0)>\t'_c(t_0).
$$
Now take some $x_0<x_c$. The orbit $\g0$ must end up converging to the line $x=0$ at some $(0,\t_0),\ \t_0\in(\pi/2,\pi)$. Otherwise, it would intersect $\gamma_c$ transversely, a contradiction. Note that it $\g0$ cannot converge to the point $(0,\pi/2)$, since we would take $x_1\in(x_0,x_c)$ and then $\gamma_{x_1}$ would intersect transversely $\gamma_c$.

In this fashion, the orbit $\gamma_\Phi$ that passes through $(0,\pi/2)$ lies at the right-hand side of $\gamma_c$ and so intersects $\t=3\pi/2$ at some $(x_\infty,3\pi/2)$, with $x_\infty>x_c>1/a$. Finally, we distinguish cases for $x_0$.
\begin{itemize}
\item If $x_0\leq x_\infty$, then $\g0$ must end up converging to the line $x=0$ at some $(0,\t_0),\ \t_0\in[\pi/2,\pi)$. This orbit generates a non-complete $\Phi$-surface.
\item If $x_0>x_\infty$, then $\g0$ intersects the line $\t=\pi/2$ at some $(\widehat{x_0},\pi/2),\ \widehat{x_0}>0$. This orbit generates a complete $\Phi$-surface with a nodoidal behavior.
\end{itemize}
\end{enumerate}

This completes the classification of rotational $\Phi$-surfaces of hyperbolic type.
\end{proof}

\begin{obs}\label{observation}
In \cite{Lop}, López considered the equivalent relation $aH+bK=1$, which is hyperbolic if and only if $a^2+4b<0$. He described the same four types of rotational surfaces in terms of a parameter $z_0\neq-2b/a$ varying as: 
\begin{itemize}
\item $0<z_0<a/2$ for the compact surfaces of positive Gauss curvature;
\item $z_0=a/2$ for the cylinder;
\item $a/2<z_0<-2b/a$ for the non-complete surfaces of negative Gauss curvature; and
\item $z_0>-2b/a$ for the complete surfaces of nodoidal-type.
\end{itemize}
It is easy to check that the value $a/2$ for the relation $aH+bK=1$ corresponds to $-a/\phi(0)$ in our case $2aH+K=\phi$, and $-2b/a$ corresponds to $1/a$. He claimed that for every $z_0>-2b/a$, the corresponding rotational surface was complete, which we already proved to be false.
\end{obs}

\def\refname{References}

The first author was partially supported by P18-FR-4049. For the second author, this research is a result of the activity developed within the framework of the Programme in Support of Excellence Groups of the Región de Murcia, Spain, by Fundación Séneca, Science and Technology Agency of the Región de Murcia. Irene Ortiz was partially supported by MICINN/FEDER project PGC2018-097046-B-I00 and Fundación Séneca project 19901/GERM/15, Spain.

\end{document}